\documentclass[a4paper,reqno]{amsart}
\usepackage[all]{xy}           
\usepackage{amssymb}           
\usepackage{eucal}
\usepackage[usenames,dvipsnames]{color}

\numberwithin{equation}{section}

\newtheorem{definition}{Definition}[section]
\newtheorem{lemma}[definition]{Lemma}
\newtheorem{theorem}[definition]{Theorem}
\newtheorem{proposition}[definition]{Proposition}

\newtheorem{remarkth}[definition]{Remark}
\newtheorem{example}[definition]{Example}
\newenvironment{remark}{\begin{remarkth}\upshape}{\hfill$\diamond$\end{remarkth}}
\renewcommand{\emph}[1]{{\bfseries\itshape{#1}}}

\usepackage{amssymb}

\newcommand{\R}{\mathbb{R}}      

\newcommand{\T}{\mathbb{T}}

\newcount\ancho \newcount\anchom \newcount\anchoa
\newcount\anchob \newcount\altura

\newcommand{\ltilde}[3][0]{\altura=0 \advance\altura by #1
	\ancho=#2 \anchom=\ancho \divide\anchom by 2
	\anchoa=\ancho \divide\anchoa by 4
	\anchob=\anchom \advance\anchob by \anchoa
	\kern-3pt \begin{array}[b]{c}
		\begin{picture}(1,1)(\anchom,-\altura)
		\qbezier(0,2)(\anchoa,5)(\anchom,2)
		\qbezier(\anchom,2)(\anchob,-1)(\ancho,4)
		\qbezier(0,2)(\anchoa,4.5)(\anchom,1.8)
		\qbezier(\anchom,1.8)(\anchob,-1.5)(\ancho,4)
		\end{picture} \\[-4pt]{#3}
	\end{array} \kern-4pt    }

\newcommand{\lhat}[3][0]{\altura=0 \advance\altura by #1
	\ancho=#2 \anchom=\ancho \divide\anchom by 2
	\anchoa=\ancho \divide\anchoa by 4
	\anchob=\anchom \advance\anchob by \anchoa
	\kern-3pt \begin{array}[b]{c}
		\begin{picture}(1,1)(\anchom,-\altura)
		\qbezier(0,2)(\anchoa,4)(\anchom,6)
		\qbezier(\anchom,6)(\anchob,4)(\ancho,2)
		\qbezier(0,2)(\anchoa,3.8)(\anchom,5.6)
		\qbezier(\anchom,5.6)(\anchob,3.8)(\ancho,2)
		\end{picture} \\[-4pt] {#3}
	\end{array} \kern-4pt    }

\newcommand{\lcf}{\lbrack\! \lbrack}
\newcommand{\rcf}{\rbrack\! \rbrack}
\newcommand\map[3]{#1\ \colon\ #2\longrightarrow#3}
\newcommand{\lvec}[1]{\overleftarrow{#1}}
\newcommand{\rvec}[1]{\overrightarrow{#1}}

\newcommand{\e}{\mathrm{e}}

\makeatletter
\newcommand\prol{\@ifstar{\@proldf}{\@prolpf}}  
\def\@prolpf{\@ifnextchar[{\@prolpf@wrt}{\@prolpf@}}
\def\@prolpf@wrt[#1]#2{\@ifnextchar[{\@prolpf@wrt@at{#1}{#2}}{\@prolpf@wrt@{#1}{#2}}}
\def\@prolpf@wrt@at#1#2[#3]{\prolsymbol^{#1}_{#3}#2}
\def\@prolpf@wrt@#1#2{\prolsymbol^{#1}#2}
\def\@prolpf@#1{\@ifnextchar[{\@prolpf@at{#1}}{\@prolpf@@{#1}}}
\def\@prolpf@at#1[#2]{\prolsymbol_{#2}#1}
\def\@prolpf@@#1{\prolsymbol#1}
\def\@proldf{\@ifnextchar[{\@proldf@wrt}{\@proldf@}}
\def\@proldf@wrt[#1]#2{\@ifnextchar[{\@proldf@wrt@at{#1}{#2}}{\@proldf@wrt@{#1}{#2}}}
\def\@proldf@wrt@at#1#2[#3]{\prolsymbol^{*#1}_{#3}#2}
\def\@proldf@wrt@#1#2{\prolsymbol^{*#1}#2}
\def\@proldf@#1{\@ifnextchar[{\@proldf@at{#1}}{\@proldf@@{#1}}}
\def\@proldf@at#1[#2]{\prolsymbol^*_{#2}#1}
\def\@proldf@@#1{\prolsymbol^*#1}
\def\prolsymbol{\mathcal{T}}
\makeatother

\def\lcf{\lbrack\! \lbrack}
\def\rcf{\rbrack\! \rbrack}
\newcommand{\cinfty}[1]{C^\infty(#1)}
\newcommand{\set}[2]{\left\{\,#1\left.\vphantom{#1#2}\,\right\vert\,#2\,
	\right\}}

\newcommand{\pd}[2]{\frac{\partial #1}{\partial #2}}

\newcommand{\vectorfields}[1]{\mathfrak{X}(#1)}

\newcommand{\sode}{{\textsc{sode}}}

\makeatletter
\@namedef{subjclassname@1991}{{\rm 2020} Mathematics Subject Classification}
\makeatother

\begin{document}
	{\Large
		
		\title[Local convexity for second order differential equations on a Lie algebroid]{Local convexity for second order differential equations on a Lie algebroid}
		
		\author[J.C. Marrero]{J.C. Marrero}
		\address{J.C. Marrero:
			ULL-CSIC Geometr\'{\i}a Diferencial y Mec\'anica Geom\'etrica\\
			Departamento de Matem\'aticas, Estad{\'\i}stica e IO, Secci\'on de
			Ma\-te\-m\'a\-ti\-cas y F{\'\i}sica, Universidad de la Laguna, La Laguna,
			Tenerife, Canary Islands, Spain} \email{jcmarrer@ull.edu.es}
		
		\author[D. Mart\'{\i}n de Diego]{D. Mart\'{\i}n de Diego}
		\address{D. Mart\'{\i}n de Diego:
			Instituto de Ciencias Matem\'aticas (CSIC-UAM-UC3M-UCM) \\ C/Nicol\'as
			Cabrera 13-15, 28049 Madrid, Spain} \email{david.martin@icmat.es}
		
		\author[E. Mart\'{\i}nez]{E. Mart\'{\i}nez}
		\address{E. Mart\'{\i}nez:
			Departamento de Matem\'atica Aplicada e IUMA, Facultad de
			Ciencias, Universidad de Zaragoza, 50009 Zaragoza, Spain}
		\email{emf@unizar.es}

		\thanks{{}}
		
		
		
		\maketitle
		
		\hfill{\small\it{Dedicated to the memory of K Mackenzie}}
		
		\begin{abstract}
			A theory of local convexity for a second order differential equation (\sode\ ) on a Lie algebroid is developed. The particular case when the \sode\ is homogeneous quadratic is extensively discussed. 
		\end{abstract}

		
		\tableofcontents

\section{Introduction}

Lie algebroids and groupoids were fundamental mathematical objects in the research by K Mackenzie. In fact, an standard reference for these topics is his book \cite{Mac}. On the other hand, the relation of Lie algebroids and groupoids with Physics is increasingly important, specially with classical and quantum mechanics and, particularly, with Lagrangian mechanics.

\noindent{\bf Lagrangian mechanics on Lie algebroids.}
Lagrangian mechanics on Lie algebroids has deserved a lot of interest in recent years from the seminal paper by Weinstein \cite{We}. The notion of a Lie algebroid \cite{Mac} allows to discuss general Lagrangian systems beyond the ones defined on the tangent bundle of the configuration manifold. These include systems determined by Lagrangian functions on Lie algebras, Atiyah algebroids, action Lie algebroids...These systems appear, in a natural way, when one applies reduction by a symmetry Lie group (see \cite{CoLeMaMaMa,LeMaMa,Ma0,Ma,We}).

The typical definition of a mechanical Lagrangian function $L: A \to \mathbb{R}$ on a Lie algebroid $\tau: A \to Q$ is
\[
L(a) = K_{g}(a) = \frac{1}{2} g(a, a), \; \; \mbox{ for } a\in A,
\]
where $g: A \times_Q A \to \mathbb{R}$ is a bundle metric on $A$ or, more generally,
\[
L(a) = K_{g}(a) - V(\tau(a)) = \frac{1}{2} g(a, a) -V(\tau(a)), \; \; \mbox{ for } a \in A,
\]
with $V \in C^{\infty}(Q)$. In the first case, $L$ represents the kinetic energy on $A$ induced by the bundle metric $g$ and, in the second case, $L$ is the difference of the kinetic energy and the potential energy induced by the real $C^{\infty}$-function $V$ on $Q$. 

In both cases, the solutions of the Euler-Lagrange equations for $L$ (the dynamical equations) are the integral curves of a second order differential equation (a \sode$\,$) $\Gamma_L$ on $A$, that is, a vector field on $A$ whose integral curves are admissible. This means that an integral curve $\gamma: I \to A$ of $\Gamma_L$ satisfies the following condition
\[
\frac{d}{dt}(\tau \circ \gamma) = \rho \circ \gamma,
\]
where $\rho: A \to TQ$ is the anchor map of $A$. Moreover, in the first case when $L$ is the kinetic energy induced by a bundle metric on $A$, $\Gamma_L$ is a homogeneous quadratic \sode\, i.e., the (local) coefficients of $\Gamma_L$ in the generalized velocities are fiberwise homogeneous quadratic functions.

\noindent{\bf Our motivation.}
The explicit integration of the \sode\ $\Gamma_L$ is, in general, quite complicated. So, of particular interest, it is the construction of geometric integrators for Lagrangian systems using a discrete variational principle. In this direction, a lot of effort has been devoted to the case when the discrete Lagrangian function is defined on the cartesian product $Q \times Q$ of a smooth manifold $Q$ (see \cite{MaWe,PaCu}). 
$Q \times Q$ plays the role of a discretized version of the standard velocity phase space $TQ$ and the discrete Lagrangian function $L_d: Q \times Q \to \mathbb{R}$ is an approximation of the continuous Lagrangian function $L: TQ \to \mathbb{R}$. The discrete evolution operator is given by a smooth map
\[
\xi: Q \times Q \to Q \times Q, \; \; (q_0, q_1) \to \xi(q_0, q_1) = (q_1, q_2)
\]
which inherits some of the geometric properties of the flow of the \sode\ $\Gamma_L$ at a fixed sufficiently small time $h > 0$. For the comparison between $\xi$ and the exact flow of $\Gamma_L$ at time $h$, it is crucial the construction of the exact discrete Lagrangian function whose discrete evolution is just the flow of $\Gamma_L$ at time $h$. For the construction of this discrete Lagrangian function, one must use (local) convexity results for the standard \sode\ $\Gamma_L$ (see \cite{MaWe,PaCu}).

Note that $TQ$ is just the Lie algebroid of the pair Lie groupoid $Q \times Q$. So, when one considers the more general case of a continuous Lagrangian function defined on the Lie algebroid $AG$ of a Lie groupoid $G$, then the Lie groupoid $G$ plays the role of the discrete phase space. In fact, a theory on discrete Lagrangian mechanics on Lie groupoids have been developed in last years (see \cite{MaMaMa,MaMaMa2,
	mo-ve,We}). For the analysis of the error, and as in the standard case when the Lie groupoid is the pair groupoid $Q \times Q$, it is very important the construction of the exact discrete Lagrangian function on the Lie groupoid $G$ associated with a continuous Lagrangian function on the Lie algebroid $AG$ of $G$. In turn, for the introduction of this discrete Lagrangian function, one must previously have a theory of local convexity for a \sode\ on $AG$.  

The aim of this paper is to provide this theory. 

\noindent{\bf The main ideas in the paper.}
For the previous purpose, we proceed as follows:
\begin{enumerate}
	\item
	We prove appropiate convexity results for a \sode\ which is defined in a special Lie algebroid: the vertical bundle of a fibration;
	\item
	We introduce a Lie algebroid morphism
	\[
	\Psi: V\alpha \to AG
	\]
	between the vertical bundle of the source map $\alpha: G \to Q$ of the Lie groupoid $G$ and $AG$, which is fiberwise bijective and
	\item
	We prove that given a \sode\ $\Gamma$ on $AG$, there exists a unique \sode\ $\tilde{\Gamma}$ on $V\alpha$ which is $\Psi$-related with $\Gamma$ and, then, we use 1.
\end{enumerate}
\noindent{\bf The structure of the paper.}
The paper is structured as follows. In Section \ref{section2}, we discuss local convexity results for an standard \sode\ on the tangent bundle of a manifold. In Section \ref{sec4}, we extend these results for the more general case when the \sode\ is defined on a Lie algebroid. In Section \ref{sec5}, we consider the particular case of a homogeneous quadratic \sode\ on a Lie algebroid. Section \ref{conclusions-future-work} contains the conclusions and some comments on the future work. The paper ends with two appendices which contain the proofs of Theorems \ref{convexity1} and \ref{isomorfismo-exp-h0-q0} and some basic results on Lie algebroids and groupoids. 

\noindent{\bf Relation with our preprint \cite{MaMaMa3}.} A part of the results in this paper (Sections \ref{section2} and \ref{sec4}) are contained in our unpublished preprint \cite{MaMaMa3}. The reason is that we have decided to split this preprint in two different parts. The first one is this paper and, as you can see, it contains the results in \cite{MaMaMa3} plus a discussion on the particular case of a homogeneous quadratic \sode\ in a Lie algebroid. Moreover, another difference is that the results in Section \ref{section2} of this paper are presented in a slightly different way which is, in our opinion, clearer and more accurate. In particular,
we include an analytical explicit proof of Theorem \ref{isomorfismo-exp-h0-q0} (see Appendix \ref{Hartmann}) which is missing in our preprint \cite{MaMaMa3}. 

On the other hand, in the second part in which we split the preprint \cite{MaMaMa3}, we will introduce the discrete exact Lagrangian function associated with a continuous regular Lagrangian function on the Lie algebroid of a Lie groupoid and, in addition, we will discuss the variational error analysis in this setting (see \cite{MaMaMa4}).    

\section{Local convexity  for standard second order differential
	equations}\label{section2}

Let $Q$ be a smooth manifold, $TQ$  its tangent bundle and $T^{2}Q$ the space of second jets of curves on $Q$. We
will denote by $\tau_{Q}: TQ \to Q$ the canonical projection. As we know, from local coordinates $(q^i)$ on $Q$,
we can induce local coordinates $(q^i, \dot{q}^i)$ on $TQ$ and $(q^i, \dot{q}^i, \ddot{q}^i)$ on $T^2Q$.

A standard \sode\ on $Q$ is a subset $S$ of $T^2Q$. If the equation $S$ is explicit, which means that 
\[
\ddot{q}^i=\xi^i (q, \dot{q}), \ \forall i\, ,
\]
then, using the canonical inclusion of $T^2Q$ in $TTQ$, 
$S$ may be transformed to a vector field $\Gamma$ on $TQ$, that is, $S=\hbox{im}\Gamma$.
In fact, the local expression of $\Gamma$ is
\begin{equation}\label{Loc-exp-Gamma}
\Gamma (q, \dot{q})=\dot{q}^i\frac{\partial}{\partial q^i}+\xi^i(q, \dot{q})\frac{\partial}{\partial \dot{q}^i}\; .
\end{equation}
Along this paper, we will assume that our \sode\ is explicit and we will refer to it as the vector field 
$\Gamma$ on $TQ$. Note that, in such a case, the integral curves of $\Gamma$ are tangent lifts of curves in $Q$: 
the trajectories of the \sode\ $\Gamma$. In fact, these curves $t\rightarrow (q^i(t))$ in $Q$ satisfy the equations
\[
\frac{d^2 q^i}{dt^2}(t)=\xi^i(q(t), \frac{dq}{dt}(t))\; ,\ \forall i\; .
\] 
Therefore, a first convexity
theorem for a \sode\   may be deduced using the theory of explicit second order differential equations. 


%
%
\begin{theorem}\label{convexity1}
	Let $\Gamma$ be a \sode\ in $Q$ and $q_{0}$ be a point of $Q$. Then,
	one may find a sufficiently small positive number $h_{0}$, a family of tangent vectors of $Q$ at $q_0$, 
	\[
	v_{(h, q_0)} \in T_{q_{0}}Q, \; \; \mbox{ for } 0 < h \leq h_0,
	\]
	and two compact subsets $C$ and $\bar{C}$ of $Q$ and $TQ$, respectively, with $q_0 \in C$ and $v_{(h, q_0)} \in \bar{C}$, such that there exists
	a unique trajectory of $\Gamma$
	\[
	\sigma_{q_0q_0h}: [0, h] \to C \subseteq Q
	\]
	satisfying
	\[
	\sigma_{q_0q_0h}(0) = q_0, \; \; \; \sigma_{q_0q_0h}(h) = q_0,
	\]
	and 
	\[
	\dot{\sigma}_{q_0q_0h}(t) \in \bar{C}, \; \mbox{ for every } t \in [0, h].
	\]
\end{theorem}
\begin{proof} A proof of this result may be found in Appendix \ref{Hartmann}.
\end{proof}

Let $\Gamma$ be a \sode\ on $Q$.

We will denote by $\Phi^{\Gamma}$ the flow of $\Gamma$
\[
\Phi^{\Gamma}: D^{\Gamma}\subseteq \mathbb{R} \times TQ \to TQ.
\]
Here, $D^{\Gamma}$ is the open subset of $\mathbb{R} \times TQ$
given by
\[
D^{\Gamma} = \{(t, v) \in \mathbb{R} \times TQ \mid \Phi^{\Gamma}(\cdot, v) \mbox
{ is defined at least in } [0, t] \}.
\]
Now, if $q_{0}$ is a point of $Q$ and $h \geq 0$, we may
consider the open subset $D^{\Gamma}_{(h, q_{0})}$ of
$T_{q_{0}}Q$ given by
\[
D^{\Gamma}_{(h,q_{0})} = \{v \in T_{q_{0}}Q \mid (h, v) \in
D^{\Gamma} \}.
\]
Note that if $h > 0$ is sufficiently small then it is clear
that $D^{\Gamma}_{(h,q_{0})} \neq \emptyset $. Moreover, we
may introduce the exponential map associated with $\Gamma$
at $q_{0}$ for the time $h$ as follows
\begin{equation}\label{Def-exp-h0-q0}
exp^{\Gamma}_{(h, q_{0})}(v) = (\tau_{Q} \circ
\Phi^{\Gamma})(h, v), \; \; \mbox{ for } v\in
D^{\Gamma}_{(h, q_{0})}.
\end{equation}
We remark that the map $exp^{\Gamma}_{(0, q_{0})}$ is constant.
However, we have the following result.
\begin{theorem}\label{isomorfismo-exp-h0-q0}
	Let $\Gamma$ be a \sode\ in $Q$ and $q_0$ a point in $Q$. We take a sufficiently small positive real number $h$ and $v_{(h, q_0)} \in T_{q_0}Q$ as in Theorem \ref{convexity1}. Then,
	\[
	v_{(h, q_0)} \in D^{\Gamma}_{(h, q_0)}, \; \; exp^{\Gamma}_{(h, q_0)}(v_{(h, q_0)}) = q_0,
	\]
	and 
	\[
	T_{v_{(h, q_0)}}exp^{\Gamma}_{(h, q_0)}: T_{v_{(h, q_0)}}D^{\Gamma}_{(h, q_0)} \to T_{q_0}Q
	\]
	is an isomorphism.
\end{theorem}
\begin{proof} A proof of this result may be found in Appendix \ref{Hartmann}.
\end{proof}
From Theorem \ref{isomorfismo-exp-h0-q0}, we have that there exist open subsets ${\mathcal U}_0$ and $U$ in $D^{\Gamma}_{(h, q_0)}$ and $Q$, respectively, with $v_{(h, q_0)} \in {\mathcal U}_0$ and $q_0 \in U$, such that  the map 
\[
exp^{\Gamma}_{(h, q_0)}: {\mathcal U}_0 \subseteq D^{\Gamma}_{(h, q_0)} \to U\subseteq Q
\]
is a diffeomorphism. 

Next, we will consider
the open subset $D^{\Gamma}_{h}$ of $TQ$ given by
\[
D^{\Gamma}_{h} = \{v \in TQ \mid (h, v) \in D^{\Gamma} \}.
\]
Note that
\[
v \in D^{\Gamma}_{h} \Longrightarrow D^{\Gamma}_{(h,
	\tau_{Q}(v))} = D^{\Gamma}_{h} \cap T_{\tau_{Q}(v)}Q \subseteq
D^{\Gamma}_{h}.
\]
Thus, since $\tau_Q: TQ \to Q$ is an open map, it follows that $\tau_Q(D^{\Gamma}_{h})$ is an open subset 
of $Q$ and
\[
D^{\Gamma}_{h} = \bigcup_{q\in \tau_Q(D^{\Gamma}_{h})} D^{\Gamma}_{(h, q)}.
\]
In addition, we may define the smooth map $exp_{h}^{\Gamma}:
D^{\Gamma}_{h} \subseteq TQ \to Q \times Q$ as follows
\[
exp^{\Gamma}_{h}(v) = (\tau_{Q}(v), exp^{\Gamma}_{(h,
	\tau_{Q}(v))}(v)), \; \; \mbox{ for } v \in D^{\Gamma}_{h}.
\]
Moreover, we deduce that
\begin{lemma}\label{non-singular}
	Let $v$ be an element of $D^{\Gamma}_{h}$ such that
	$exp^{\Gamma}_{(h, \tau_{Q}(v))}$ is non-singular at $v$.
	Then, $exp^{\Gamma}_{h}$  is also non-singular at $v$.
\end{lemma}
\begin{proof}
	We must prove that the map
	\[
	T_{v}(exp^{\Gamma}_{h}): T_v(D^{\Gamma}_{h}) \simeq
	T_v(TQ) \to T_{(\tau_{Q}(v), exp^{\Gamma}_{(h,
			\tau_{Q}(v))}(v))}(Q\times Q) \simeq T_{\tau_{Q}(v)}Q \times
	T_{exp^{\Gamma}_{(h,\tau_{Q}(v))}(v)}Q
	\]
	is a linear isomorphism.
	
	Suppose that
	\[
	0 = (T_{v}(exp_{h}^{\Gamma}))(X_{v}), \; \; \mbox{ with }
	X_{v} \in T_{v}(D^{\Gamma}_{h}).
	\]
	Then, we have that
	\[
	0 = (T_{v}\tau_{Q})(X_{v}) \; \; \mbox{ and } \; \; 0 =
	(T_{v}exp^{\Gamma}_{(h,\tau_{Q}(v))})(X_{v}).
	\]
	The first condition implies that
	\[
	X_{v} \in T_{v}(D^{\Gamma}_{h} \cap T_{\tau_{Q}(v)}Q) =
	T_{v}(D^{\Gamma}_{(h, \tau_{Q}(v))})
	\]
	and thus, using the second one, we conclude that
	\[
	X_{v} = 0.
	\]
\end{proof}
As we know, if $h > 0$ is sufficiently small and $q_{0} \in Q$
then the map $exp_{(h, q_{0})}^{\Gamma}: D^{\Gamma}_{(h,
	q_{0})} \to Q$ is non-singular at the point
$v_{(h, q_0)} \in D^{\Gamma}_{(h,
	q_{0})}$. Therefore, using Lemma \ref{non-singular}, we deduce the following result

\begin{theorem}\label{convexity2}
	Let $\Gamma$ be a \sode\ in $Q$ and $q_{0}$ be a point of $Q$. Then,
	one may find a sufficiently small positive number $h$, an open
	subset ${\mathcal U} \subseteq D^{\Gamma}_{h} \subseteq TQ$, with $v_{(h, q_0)}
	\in {\mathcal U}$, and an open subset $U$ of $Q$, with $q_{0} \in
	U$, such that:
	\begin{enumerate}
		\item
		The exponential map associated with $\Gamma$ at time $h$ 
		\[
		exp^{\Gamma}_{h}: {\mathcal U} \subseteq D^{\Gamma}_{h} \to U \times U \subseteq Q \times Q
		\]
		is a diffeomorphism.
		\item
		For every couple $(q, q') \in U \times U$ there
		exists a unique trajectory of $\Gamma$
		\[
		\sigma_{qq'h}: [0, h] \to Q
		\]
		satisfying
		\[
		\sigma_{qq'h}(0) = q, \; \; \sigma_{qq'h}(h) = q' \; \; \mbox{
			and } \; \; \dot{\sigma}_{qq'h}(0) \in {\mathcal U}.
		\]
	\end{enumerate}
\end{theorem}

We will denote by 
$R^{e^-}_{h}: U \times U \to {\mathcal U}$ (respectively, $R^{e+}_{h}: U \times U \to {\mathcal U}$)
the inverse map of the diffeomorphism 
$exp^{\Gamma}_{h}: {\mathcal U}  \to U \times U $ (respectively, $exp^{\Gamma}_{h} \circ \Phi^{\Gamma}_{-h}: 
\Phi^{\Gamma}_{h}({\mathcal U}) \to U \times U$).

The maps 
\[ 
R^{e^-}_{h}: U \times U \subseteq Q \times Q \to {\mathcal U}\subseteq TQ \mbox{ and }
R^{e^+}_{h}: U \times U \subseteq Q \times Q \to \Phi^{\Gamma}_{h}({\mathcal U})\subseteq TQ
\]
are called the exact retraction maps associated with $\Gamma$. We have that
\[
R^{e^-}_{h}(q, q') = \dot{\sigma}_{qq'h}(0), \; \; R^{e^+}_{h}(q, q') = \dot{\sigma}_{qq'h}(h).
\]
Note that
\[
R^{e^+}_{h} = \Phi^{\Gamma}_{h} \circ R^{e^-}_{h},
\]
that is, the following diagram
\[
\xymatrix{ {\mathcal U}\subseteq TQ
	\ar[dd]_{\Phi_{h}^\Gamma} &   & &U\times U\subseteq Q\times Q \ar[lll]_{R^{e^-}_{h}}\ar[ddlll]^{{R^{e^+}_{h}}}\\
	&  & &\\
	\Phi_{h}^\Gamma( {\mathcal U})\subseteq TQ&  & & }
\]
is commutative.



\let\s\alpha
\let\t\beta
\let\e\epsilon
\def\rightrightarrow{\rightrightarrows}
\def\T{\CMcal{T}}


\begin{remark}\label{reweyls}
	{\rm 
		In some applications, it is useful to define the following map
		\begin{equation}\label{weyl}
		\widetilde{exp}^{\Gamma}_{h}(v)=(exp^{\Gamma}_{(-h/2, \tau_Q(v))}(v), exp^{\Gamma}_{(h/2, \tau_Q(v))}(v))
		\end{equation}
		which is a local diffeomorphism since
		\[
		\widetilde{exp}^{\Gamma}_{h}=exp_{h}^{\Gamma}\circ \Phi^{\Gamma}_{-h/2}\; .
		\]
	}
\end{remark}

\section{Convexity theorems for second order differential equations on Lie algebroids}\label{sec4}

In this section, we will obtain a version of Theorem \ref{convexity2} for a \sode\  on a general Lie algebroid
$A$.

First of all, we will recall the definition of a \sode\ on $A$ (see, for instance, \cite{LeMaMa}).

Let $\Gamma$ be a vector field on the tangent bundle $TQ$ of a manifold $Q$. Then, using (\ref{Loc-exp-Gamma}), it is easy to prove that $\Gamma$ is a \sode\ if and only if
\[
(T_v\tau_Q)(\Gamma(v)) = v, \; \; \mbox{ for } v \in TQ.
\]
This fact suggest us to introduce the definition of a \sode\ on the Lie algebroid $\tau: A \to Q$, with Lie algebroid structure $(\lcf \cdot, \cdot \rcf, \rho)$ (see Appendix \ref{algebroide-grupoide}), as follows. A vector field $\Gamma$ on $A$ is a \sode\ if
\[
(T_a\tau)(\Gamma(a)) = \rho(a), \; \; \mbox{ for } a\in A,
\]
and, in such a case, the trajectories of $\Gamma$ are the projections, via $\tau: A\to Q$, of the integral curves of $\Gamma$. In other words,
\begin{equation}\label{SODE-2}
\Gamma \mbox{ is a \sode\ on } A \Leftrightarrow \Gamma(f \circ \tau) = \widehat{d^{A}f}, \; \; \mbox{ for } f \in C^{\infty}(Q),
\end{equation}
where $\widehat{d^{A}f}$ is the fiberwise linear function on $A$ given by
\[
\langle \widehat{d^{A}(f)}, a \rangle = \rho(a)(f), \; \; \mbox{ for } a \in A.
\] 
Let $(q^{i})$ be local coordinates on an open subset $U \subseteq Q$ and $\{e_{\alpha}\}$ a local basis of $\Gamma(A)$, the space of sections of $A$, such that
\[
\rho(e_\alpha) = \rho^{i}_\alpha(q) \frac{\partial}{\partial q^{i}}.
\]
Denote by $(q^{i}, y^{\alpha})$ the corresponding local coordinates on $A$. Then, the local expression of a \sode\ $\Gamma$ is
\[
\Gamma(q, y) = \rho^{i}_\alpha(q) y^\alpha \frac{\partial}{\partial q^{i}} + \Gamma^{\alpha}(q, y) \frac{\partial}{\partial y^{\alpha}}.
\]
So, the integral curves of $\Gamma$ satisfy the following system of differential equations
\[
\displaystyle \frac{dq^{i}}{dt} = \rho^{i}_\alpha(q) y^{\alpha}, \; \; \frac{dy^{\alpha}}{dt} = \Gamma^{\alpha}(q, y).
\]
Thus, we have not an explicit system of second order differential equations as in the case of an standard \sode\ on $TQ$. However, we will prove some local convexity theorems for the \sode\ $\Gamma$ on $A$.  

In fact, in Section \ref{fibration}, we will discuss the particular case when $A$ is a special
integrable Lie algebroid: the vertical bundle associated with a fibration. In such a case, we will
prove a parametrized version of Theorem \ref{convexity2}. Next, in Section \ref{integrable}, we will consider
the more general case when $A$ is the Lie algebroid $AG$ associated with an arbitrary Lie
groupoid $G$. In fact, we will see that a \sode\ on $AG$ induces a \sode\ on the vertical
bundle of the source map of $G$ and, then, we will apply the results of the previous section. Finally, we will 
discuss the general case. For this purpose, we will use that for every Lie algebroid $A$ there exists a
(local) Lie groupoid whose Lie algebroid is $A$.

\subsection{The particular case of the vertical bundle of a fibration}\label{fibration}

Consider a surjective submersion $\map{\pi}{Q}{M}$ and the pair
groupoid $Q\times Q\rightrightarrow Q$ (see Appendix \ref{algebroide-grupoide}) . The subset $G\pi\subset Q\times Q$
given by
\[
G\pi=\set{(q_1,q_2)\in Q\times Q}{\pi(q_1)=\pi(q_2)}
\]
is a Lie subgroupoid of the pair groupoid. In consequence, the source and target
maps are $(q_1,q_2)\mapsto q_1$ and $(q_1,q_2)\mapsto q_2$,
respectively, the identity map is $q\mapsto (q,q)$ and the
multiplication is given by $(q_1,q_2)(q_2,q_3)=(q_1,q_3)$. The Lie
algebroid of $G\pi$ is the vector bundle whose fiber at a point
$q$ is
\[
A_q(G\pi)=\set{(0,v)\in T_qQ\times T_qQ}{T\pi(v)=0}
\]
with the anchor $(0,v)\mapsto v$, and hence it can be identified
with the vertical bundle $\map{\tau}{V\pi=\ker{T\pi}}{Q}$, with the
canonical inclusion as the anchor map and Lie bracket in the space 
of sections $\Gamma(\tau)$ the restriction to $\Gamma(V\pi)$ of 
the standard Lie bracket of vector fields.

On $V\pi$ we can take coordinates as follows. We consider local
coordinates $(q^i)=(q^a,q^\alpha)$ in $Q$ adapted to the submersion
$\pi$, that is, $\pi(q^a,q^\alpha)=(q^a)$. The coordinate vector
fields $\{e_\alpha=\partial/\partial q^\alpha\}$ are a basis of
local sections of $V\pi$, and hence we have coordinates
$(q^a,q^\alpha,y^\alpha)$ on $V\pi$, where $y^\alpha$ are the
components of a vector on such a coordinate basis. Thus, if $(\lcf \cdot, \cdot \rcf, \rho)$ is the Lie algebroid structure on $V\pi$, it follows that
\[
\rho (e_\alpha) = \frac{\partial}{\partial q^\alpha} \;\; \mbox{ and } \; \; \lcf e_\alpha, e_\beta \rcf = 0.
\]
This implies that the structure functions are
\[
\rho^a_\alpha=0\qquad
\rho^\beta_\alpha=\delta_\alpha^\beta\qquad\text{and}\qquad
C^\alpha_{\beta\gamma}=0.
\]

A \sode\ vector field on $V\pi$ is a vector field
$\Gamma\in\vectorfields{V\pi}$ such that
$T_a\tau(\Gamma(a))=\rho(a)$ for every $a\in V\pi$. In other words
$T_a\tau(\Gamma(a))$ is the vertical vector $a\in V\pi$ itself, and
hence, if $m=\pi(\tau(a))$ then the vector $\Gamma(a)$ is tangent
to $T(\pi^{-1}(m))$. It follows that a \sode\ on $V\pi$ is but a
parametrized version of an ordinary \sode, where the parameters
are the coordinates on the base manifold $M$. In other words,  it is
a smooth family of ordinary \sode s, one on each fiber of the
projection $\pi: Q \to M$.

This can also be easily seen in coordinates. Locally,  such a
\sode\ vector field is
\begin{align*}
\Gamma&=\rho^a_\alpha y^\alpha\pd{}{q^a}+\rho^\beta_\alpha y^\alpha\pd{}{q^\beta}+f^\alpha(q^b,q^\beta,y^\beta)\pd{}{y^\alpha}\\
&=y^\alpha\pd{}{q^\alpha}+f^\alpha(q^b,q^\beta,y^\beta)\pd{}{y^\alpha}
\end{align*}
for some local functions $f^\alpha\in\cinfty{V\pi}$. The integral
curves of $\Gamma$ are the solutions of the differential equations
\[
\dot{q}^a=0,\qquad \dot{q}^\alpha=y^\alpha,\qquad
\dot{y}^\alpha=f^\alpha(q^b,q^\beta,y^\beta)
\]
or in other words
\begin{align*}
\dot{q}^a&=0\\
\ddot{q}^\alpha&=f^\alpha(q^b,q^\beta,\dot{q}^\beta).
\end{align*}
From this expression it is obvious that a \sode\ on $V\pi$ is
locally a parametrized version of an ordinary \sode, where the
parameters are the coordinates $(q^a)$ on the base manifold $M$.

On the other hand, $V\pi$ is a regular submanifold of $TQ$. In
fact, the canonical inclusion $i_{V\pi}: V\pi \to TQ$ is an
embedding. In addition, if $q_{0}$ is a point of $Q$ then it is
easy to prove that:
\begin{enumerate}
	\item
	There exists an open subset $W$ of $Q$, with $q_{0} \in W$, and
	there exists an standard local \sode\ $\bar{\Gamma}$ on $Q$, which
	is defined in $TW$, such that
	\[
	\bar{\Gamma}_{|TW\cap V\pi} = \Gamma_{|TW\cap V\pi}
	\]
	\item
	If $\bar{\sigma}: [0, h] \to W\subseteq Q$ is a trajectory of
	$\bar{\Gamma}$ and $\pi(\bar{\sigma}(0)) = \pi(\bar{\sigma}(h)) =
	m$ then $\bar{\sigma}([0, h]) \subseteq \pi^{-1}(m)$ and
	$\bar{\sigma}$ is a trajectory of the standard \sode\
	$\Gamma_{|T(\pi^{-1}(m))}$.
\end{enumerate}
Note that the local equations  defining  $V\pi$ as a
submanifold of $TQ$ are $y^a = 0$ and, thus, it is
sufficient to take
\[
\bar{\Gamma} = \displaystyle y^{a} \frac{\partial}{\partial
	q^a} + y^{\alpha} \frac{\partial}{\partial q^{\alpha}} +
f^{\alpha}(q^b, q^{\beta}, y^{\beta})\frac{\partial}{\partial
	y^{\alpha}}.
\]
Moreover, if $t\in [0,h]\rightarrow \bar{\sigma}(t)=(q^a(t), q^{\alpha}(t))\in Q$ is a trajectory of $\bar{\Gamma}$, 
$\pi (\bar{\sigma}(0))=\pi (\bar{\sigma}(h))=m$ and $m$ has local coordinates $(q^{a}_0)$ then, 
using that the trajectories of $\bar{\sigma}$ satisfy the local equations:
\[
\frac{d^2 q^a}{dt^2}=0, \qquad \frac{d^2 q^\alpha}{dt^2}=f^{\alpha}(q, y)
\]
and $q^a(0)=q^a(h)=q^a_0$, we deduce that
\[
q^a(t)=q^a_0,\ \forall t
\]
and
\[
\bar{\sigma}(t)=(q^a_0, q^{\alpha}(t))\in \pi^{-1}(m), \ \forall t
\]
Thus,
\[
\dot{\bar{\sigma}}(t)=(q^a_0, q^{\alpha}(t); 0, \frac{dq^{\alpha}}{dt}(t))
\]
is an integral curve of $\Gamma_{|T(\pi^{-1}(m))}$ and $\bar{\sigma}$ is a trajectory of $\Gamma_{|T(\pi^{-1}(m))}$.

Therefore, using Theorem \ref{convexity1} and the previous items 1 and 2,  one may find an open
neighborhood of $q_0$ in $Q$ and a unique curve
$\bar{\sigma}_{q_0q_0h} \equiv \sigma_{q_0q_0h}$ on it which
connects the point $q_0$ with itself and such that it is
trajectory of $\bar{\Gamma}$, for $h$ enough small. Then, from the second condition, it
follows that the curve $\sigma_{q_0q_0h}$ is contained in
$\pi^{-1}(\pi(q_{0}))$ and $v_{(h, q_0)} = \dot{\sigma}_{q_0q_0h}(0) \in
V_{q_0}(\pi)$. Moreover, if we apply Theorem \ref{convexity2} to
the standard \sode\ $\bar{\Gamma}$, we deduce the following result
\begin{theorem}
	\label{caso.Vpi} Let $\Gamma$ be a \sode\ on the Lie algebroid
	$V\pi\to Q$ and let $q_0$ be a point in $Q$. Then, there exists a
	sufficiently small positive number $h > 0$,  an open subset
	${\mathcal U} \subseteq V\pi$, with $v_{(h, q_{0})} \in {\mathcal U}$,
	and an open subset $U$ of $Q$, with $q_{0} \in U$, such that 
	\begin{enumerate}
		\item
		The exponential map of $\Gamma$ at time $h$
		\[
		exp^{\Gamma}_{h}: {\mathcal U} \to (U \times U)\cap G\pi, \; \;  v \in {\mathcal U} \to (\tau_{V\pi}(v), \tau_{V\pi}(\Phi^{\Gamma}_{h}(v))),
		\]
		is a diffeomorphism. Here, $\tau_{V\pi}: V\pi \to Q$ is the canonical projection and $\Phi^{\Gamma}$ is the flow of the vector field $\Gamma$ on $V\pi$. 
		\item
		For
		every couple $(q, q') \in (U \times U) \cap G{\pi}$ there exists
		a unique trajectory $\sigma_{qq'h}: [0, h] \to \pi^{-1}(\pi(q))$
		of the \sode\  $\Gamma_{|T(\pi^{-1}(\pi(q)))}$ which satisfies
		\[
		\sigma_{qq'h}(0) = q, \; \; \sigma_{qq'h}(h) = q' \; \; \mbox{
			and } \dot{\sigma}_{qq'h}(0) \in {\mathcal U}.
		\]
	\end{enumerate}
\end{theorem}

As in the standard case, we will denote by
\[
R^{e^-}_{h}: (U \times U)\cap G\pi \subseteq Q \times Q \to {\mathcal U}\subseteq V\pi  \mbox{ and }
R^{e^+}_{h}: (U \times U)\cap G\pi \subseteq Q \times Q \to \Phi^{\Gamma}_{h}({\mathcal U})\subseteq V\pi
\]
the exact retraction maps associated with $\Gamma$. In other words,
\[
R^{e^-}_{h} = (exp^{\Gamma}_{h})^{-1} \; \; \mbox{ and }  \; \; R^{e^+}_{h} = \Phi_{h}^{\Gamma} \circ R^{e-}_{h}.
\]

\subsection{The case of the Lie algebroid of a Lie groupoid}\label{integrable}

We will consider a Lie groupoid 
$G\rightrightarrow Q$ with source map
$\alpha$, target map $\beta$, and consider the fibration
$\pi\equiv\alpha$ and the associated Lie algebroid $V\alpha$ as
above. Let $\map{\tau}{AG}{Q}$ be the Lie algebroid of $G$, and denote
by $\rho$ its anchor (see Appendix \ref{algebroide-grupoide}).
Denote by $\Psi$ the vector bundle map
$\map{\Psi}{V\alpha}{AG}$ given by $\Psi(v_g)=Tl_{g^{-1}}v_g$,
for every $v_g\in V\alpha$, where $l_{g^{-1}}: \alpha^{-1}(\alpha(g)) \to  \alpha^{-1}(\beta(g))$ is the left-translation by $g^{-1} \in G$. This map is well defined since
$T\alpha(Tl_{g^{-1}}v_g)=0$ and hence
$Tl_{g^{-1}}v_g$ is $\alpha$-vertical at the identity in
$\beta(g)$. The following commutative diagram illustrates the situation:
\[
\xymatrix{ V\alpha \ar[rr]^{\Psi}
	\ar[dd]_{\tilde{\tau}}&    &AG \ar[dd]^{\tau}\\
	&  & \\
	G \ar[rr]^{\beta}&  & Q }
\]

Moreover, $\Psi: V\alpha \to AG$ is a Lie algebroid morphism over $\beta: G \to Q$. This follows using that
if $X \in \Gamma(AG)$ and $\lvec{X}$ is the corresponding left invariant vector field on  
$G$, then $\lvec{X}$ is a section of the vector bundle $\tilde{\tau}: V\alpha \to G$ and, in fact, the space of sections of this vector bundle is locally generated by the left-invariant vector fields on $G$. In addition:
\begin{enumerate}
	\item
	$\lvec{X}$ and $X$ are $\Psi$-related;
	\item
	If $\lcf \cdot, \cdot \rcf $ is the Lie bracket in $\Gamma(AG)$
	\[
	\lvec{\lcf X, Y \rcf} = [\lvec{X}, \lvec{Y}], \; \; \; \mbox{ for } X, Y \in \Gamma(AG)
	\]
	and
	\item
	The vector field $\lvec{X}$ on $G$ is $\beta$-projectable over $\rho(X)$, where $\rho$ is the anchor map
	of the Lie algebroid $AG$.
\end{enumerate}
(see Appendix \ref{algebroide-grupoide}).

On the other hand, given a \sode\ $\Gamma$ in $AG$ there exists a unique \sode\
$\tilde{\Gamma}$ in $V\alpha$ which is $\Psi$ related to $\Gamma$,
that is $T\Psi\circ\tilde{\Gamma}=\Gamma\circ\Psi$. Indeed, this
is a special case of the following result, by taking into account
that $\Psi$ is a fiberwise bijective morphism of Lie algebroids.

\begin{proposition}\label{rel-SODES}
	Let $\map{\tau_1}{E_1}{Q_1}$ and $\map{\tau_2}{E_2}{Q_2}$ be Lie
	algebroids and let $\map{\Psi}{E_1}{E_2}$ be a morphism of Lie
	algebroids which is fiberwise bijective. Given a \sode\ vector
	field $\Gamma_2$ on the Lie algebroid $E_2$ there exists a unique
	\sode\ vector field $\Gamma_1$ on the Lie algebroid $E_1$ such
	that $T\Psi\circ\Gamma_1=\Gamma_2\circ\Psi$.
\end{proposition}
\begin{proof}
	We have to show that for every $a_1\in E_1$, there exists a unique  
	$v_1\in T_{a_1}E_1$ satisfying the equations
	\[
	T\Psi(v_1)=\Gamma_2(\Psi(a_1)),\qquad\text{and}\qquad
	T\tau_1(v_1)=\rho_1(a_1).
	\]
	Note that if $v_{1}, v_{1}' \in T_{a_{1}}E_{1}$ satisfy these
	conditions then $v_{1}' - v_{1} \in Ker (T\tau_{1})$ and, since
	$\Psi$ is fiberwise bijective and $(T\Psi)(v_1) =
	(T\Psi)(v_{1}')$, we conclude that $v_{1}' = v_{1}$.
	
	Next, we will see that one may find a vector $v_{1} \in
	T_{a_{1}}E_{1}$ which satisfies the above equations.
	
	For that consider a fixed (but
	arbitrary) auxiliary \sode\ vector field
	$\Gamma\in\vectorfields{E_1}$. Since $\Gamma(a_1)$ projects to
	$\rho_1(a_1)$, the vector $v_1$ satisfies the second equation if
	and only if the vector $w_1=v_1-\Gamma(a_1)$ is vertical. If $\xi^V: E_1\times E_1\rightarrow V\tau_1$ is the canonical vertical lift, it follows that  we can write  $\xi^V(a_1,c_1)=w_1=v_1-\Gamma(a_1)$ for
	a unique $c_1\in E_1$, and then, using that $\Psi$ is a morphism
	of Lie algebroids, the first equation reads
	\[
	\xi^V(\Psi(a_1),\Psi(c_1))=\Gamma_2(\Psi(a_1))-T\Psi(\Gamma(a_1)).
	\]
	The right hand side of this equation is vertical at the point
	$\Psi(a_1)$, and since $\Psi$ is fiberwise bijective it has a
	unique solution $c_1$. Thus, the vector
	$v_1=\Gamma(a_1)+\xi^V(a_1,c_1)$ is the solution for our
	equations.
\end{proof}

Let $\Gamma$ be a \sode\ in the Lie algebroid $AG$ of a Lie groupoid $G$. Denote by $\tilde{\Gamma}$ the unique \sode\ in $V\alpha$ which is $\Psi$-related with $\Gamma$. Note that the natural inclusion 
\[
\iota: AG \to V\alpha
\]
is a Lie algebroid morphism over the identity map $\epsilon: Q \to G$. So, $AG$ is a Lie subalgebroid of the Lie algebroid $V\alpha$. However, the restriction of a \sode\ on $V\alpha$ to $AG$ is not, in general, tangent to $AG$.
\begin{example}{\rm
		Suppose that $G$ is a Lie group. Then, the Lie algebra $\frak{g}$ of $G$ is the Lie algebroid of the Lie groupoid $G$ (see Appendix \ref{algebroide-grupoide}). Moreover, since the source map $\alpha: G \to \{\frak{e}\}$ is the trivial map (with $\frak{e}$ the identity element in $G$), we have that $V\alpha$ is just the tangent bundle $TG$ of $G$ and the Lie algebroid $V\alpha \to G$ is the standard Lie algebroid $\tau_G: TG \to G$. In addition, using the left trivialisation of $TG$
		\[
		L: TG \to G \times \frak{g}, \; \; v_g \in T_gG \to (g, \Psi(v_g)),
		\]
		we can identify $TG$ with the product $G \times \frak{g}$. Under this identification, the natural inclusion $\iota: \frak{g} \to TG$ is given by
		\[
		\iota: \frak{g} \to G \times \frak{g}, \; \; \xi \to \iota(\xi) = (\frak{e}, \xi).
		\]
		Moreover, using the same identification, we have that
		\[
		TTG \simeq T(G \times \frak{g}) \simeq (G \times \frak{g}) \times (\frak{g} \times \frak{g})
		\]
		and an arbitrary \sode\ $\tilde{\Gamma}$ on $\tau_G: TG \simeq G \times \frak{g} \to G$,
		\[
		\tilde{\Gamma}: TG \simeq G \times \frak{g} \to TTG \simeq (G \times \frak{g}) \times (\frak{g} \times \frak{g}),
		\]
		has the following expression
		\[
		\tilde{\Gamma}(g, \xi) = (g, \xi; \xi, \eta(g, \xi)) \in (G \times \frak{g}) \times (\frak{g} \times \frak{g}), \; \; \mbox{ for } (g, \xi) \in G \times \frak{g}.
		\]
		On the other hand, an arbitrary vector field $\Gamma$ on ${\mathfrak g}$
		\[
		\xi\in {\mathfrak g}\longrightarrow \Gamma(\xi)=(\xi, \eta(\xi))\in {\mathfrak g}\times {\mathfrak g}
		\]
		is a \sode\ and it is clear that the corresponding \sode\ $\tilde{\Gamma}$ on $G$ is given by
		\[
		\tilde{\Gamma} (g, \xi)=(g, \xi; \xi, \eta(\xi)), \hbox{ for } (g, \xi)\in {\mathfrak g}\times {\mathfrak g}
		\]
		Thus, it is clear that the restriction of $\tilde{\Gamma}$ to the Lie subalgebroid $\frak{g} \simeq \{\frak{e}\} \times \frak{g} \subseteq G \times \frak{g} \simeq TG$ is not tangent to $\frak{g}$.
	}
\end{example}
Now, we will apply Theorem \ref{caso.Vpi} to the \sode\ $\tilde{\Gamma}$ and the point $\epsilon(q_0) \in G$, with $q_0 \in Q$. Then, one 
may find an open neighborhood of $\varepsilon(q_0)$
in $G$ and a unique curve
$\tilde{\sigma}_{\varepsilon(q_{0})\varepsilon(q_{0})h}$ on it which
connects the point $\varepsilon(q_{0})$ with itself at time $h > 0$ and such that it
is a trajectory of $\tilde{\Gamma}_{|T(\alpha^{-1}(q_0))}$. In fact,
the curve $\tilde{\sigma}_{\varepsilon(q_{0})\varepsilon(q_{0})h}$ is
contained in $\alpha^{-1}(q_{0})$ and, therefore, $v_{(h, q_{0})} =
\dot{\tilde{\sigma}}_{\varepsilon(q_{0})\varepsilon(q_{0})h}(0) \in
V_{\varepsilon(q_{0})}\alpha = A_{q_{0}}G$.

Moreover, we may prove the following result

\begin{theorem}\label{convexity-definitivo}
	Let $\Gamma$ be a \sode\ vector field on the Lie algebroid $AG\to
	Q$ of the Lie groupoid $G\rightrightarrow Q$,  $q_0\in Q$  a point in the
	base manifold and $\tilde{\Gamma}$ the corresponding \sode\ in the Lie algebroid $V\alpha \to G$. 
	Then, there exists a sufficiently small positive
	number $h > 0$, an open subset ${\mathcal U}$ in $AG$, with
	$v_{(h, q_0)} \in {\mathcal U}$, and an open subset $U$ of $G$, with
	$\varepsilon(q_{0}) \in U$, such that:
	\begin{enumerate}
		\item
		The exponential map associated with $\Gamma$ at time $h$
		\[
		exp^{\Gamma}_h: {\mathcal U} \to U, \; \; v \in {\mathcal U} \to \tilde{\tau}(\Phi^{\tilde{\Gamma}}_h(v)) \in U
		\]
		is a diffeomorphism. Here $\tilde{\tau}: V\alpha \to G$ is the canonical projection and $\Phi^{\tilde{\Gamma}}$ is 
		the flow of the vector field $\tilde{\Gamma}$ on $V\alpha$. 
		\item
		For every $g\in U$ there exists
		a unique trajectory $\sigma_{\varepsilon(\alpha(g))gh}: [0, h] \to
		\alpha^{-1}(\alpha(g))$ of $\tilde{\Gamma}$ satisfying the following conditions
		\[
		\sigma_{\varepsilon(\alpha(g))gh}(0) = \varepsilon(\alpha(g)), \; \;
		\sigma_{\varepsilon(\alpha(g))gh}(h) = g, \; \;
		\dot{\sigma}_{\varepsilon(\alpha(g))gh}(0) \in {\mathcal U}.
		\]
		Thus, the induced curve $a_{\varepsilon(\alpha(g))gh} = \Psi \circ
		\dot{\sigma}_{\varepsilon(\alpha(g))gh}$ in $AG$ is an integral curve of
		the \sode\ $\Gamma$. Moreover, the trajectory of $\Gamma$
		\[
		q_{\alpha(g)\beta(g)h} = \tau \circ a_{\epsilon(\alpha(g))gh}: [0, h] \to Q
		\]
		has initial point $\alpha(g)$ and final point $\beta(g)$, that is,
		\[
		q_{\alpha(g)\beta(g)h}(0) = \alpha(g), \; \; q_{\alpha(g)\beta(g)h}(h) = \beta(g).
		\]
	\end{enumerate}
\end{theorem}
\begin{proof}
	Using Theorem \ref{caso.Vpi}, we deduce that there exists a sufficiently small positive number $h >0$, an open subset
	$\tilde{\mathcal U}$ in $V\alpha$, with $v_{(h, q_0)} \in \tilde{\mathcal U}$, and an open subset $V$ of $G$, 
	with $\varepsilon(q_0) \in V$, such that:
	\begin{enumerate}
		\item
		The exponential map of $\tilde{\Gamma}$ at $h$
		\[
		exp_{h}^{\tilde{\Gamma}}: \tilde{\mathcal U} \subseteq V\alpha \to (V \times V) \cap G\alpha \subseteq G \times G, \; \; \tilde{v} \to (\tilde{\tau}(\tilde{v}), \tilde{\tau}(\Phi^{\tilde{\Gamma}}_{h}(\tilde{v}))),
		\]
		is a diffeomorphism.
		\item
		For every $g, g' \in (V \times V)\cap G\alpha$, there exists a unique trajectory $\sigma_{gg'h}: [0, h] \to \alpha^{-1}(\alpha(g))$ of $\tilde{\Gamma}$ satisfying the following conditions
		\[
		\sigma_{gg'h}(0) = g, \; \; \; \sigma_{gg'h}(h) = g', \; \; \; \dot{\sigma}_{gg'h}(0) \in \tilde{\mathcal U}.
		\]
	\end{enumerate}
	Now, we take the open subset $\tilde{\mathcal U} \cap AG$ of $AG$. It is clear that $v_{(h, q_0)} \in \tilde{\mathcal U}\cap AG$. 
	
	Denote by $\iota: AG \to V\alpha$ the canonical inclusion. Then, the exponential map $exp_h^{\Gamma}: \tilde{\mathcal U}\cap AG \subseteq AG \to G$ is given by
	\[
	exp^{\Gamma}_h = pr_2 \circ exp_{h}^{\tilde{\Gamma}} \circ \iota,
	\]
	where $pr_2: (V\times V)\cap G\alpha \to V\subseteq G$ is the canonical projection on the second factor. In fact,
	\begin{equation}\label{expresion-expo-tilde}
	exp_{h}^{\tilde{\Gamma}}(\iota(v)) = (\varepsilon(\tau(v)), exp_{h}^{\Gamma}(v)).
	\end{equation}
	Next, we will see that the map $exp_h^{\Gamma}: \tilde{\mathcal U} \cap AG \to G$ is a local diffeomorphism.
	Suppose that $X_v \in T_v(\tilde{\mathcal U} \cap AG)$, with $v \in \tilde{\mathcal U} \cap AG$, and
	\[
	0 = (T_v exp_h^{\Gamma})(X_v) = (T_v(\tilde{\tau} \circ \Phi_{h}^{\tilde{\Gamma}}))(X_v).
	\]
	This implies that
	\[
	(T_v(\alpha \circ \tilde{\tau} \circ \Phi_{h}^{\tilde{\Gamma}}))(X_v) = 0.
	\]
	But, since the trajectory of $\tilde{\Gamma}$ over a point $q$ of $\tau(\tilde{\mathcal U}\cap AG)$ is contained in the fiber
	$\alpha^{-1}(q)$, we deduce that
	\[
	\alpha \circ \tilde{\tau} \circ \Phi_{h}^{\tilde{\Gamma}} \circ \iota = \tau.
	\]
	Thus, we have that
	\[
	(T_v \tau)(X_v) = 0.
	\]
	Therefore, from (\ref{expresion-expo-tilde}), we deduce that
	\[
	(T_v(exp_h^{\tilde{\Gamma}} \circ \iota)) (X_v) = 0
	\]
	and it follows that $X_v = 0$.
	
	We conclude that there exists an open subset ${\mathcal U}' \subseteq AG$, with $v_{(h, q_0)} \in {\mathcal U}'$, and an open
	subset $U' \subseteq G$, such that $\varepsilon (q_0) \in U'$ and
	\[
	exp_h^{\Gamma}: {\mathcal U}' \subseteq AG \to U' \subseteq G
	\]
	is a diffeomorphism.
	
	Next, using that $\varepsilon: Q \to G$ is a continuous map, we have that there exists an open subset $W$ of $Q$ such that 
	$q_0 \in W$ and $\varepsilon(W) \subseteq U'$. So, $\alpha^{-1}(W)$ and $U = U' \cap \alpha^{-1}(W)$ are open subsets of $G$ and
	\[
	\varepsilon (q_0) \in U, \; \; \varepsilon(\alpha(U))\subseteq U.
	\]
	Thus, we may take
	\[
	{\mathcal U} = (exp_h^{\Gamma})^{-1}(U) \subseteq AG, 
	\]
	and (i) and the first part of (ii) in the theorem hold.
	
	Finally, using that the \sode\ $\tilde{\Gamma}$ is $\Psi$-related with the \sode\  $\Gamma$ and the fact that $\tau \circ \Psi = \beta \circ \tilde{\tau}$, we deduce the last part of the theorem.
\end{proof}
\begin{remark}
	The conditions satisfied by the curves $\sigma_{\varepsilon(\alpha(g))gh}$ and $a_{\varepsilon(\alpha(g))gh}$ in the previous theorem can be interpreted
	in terms of $AG$-homotopy of paths (see~\cite{CrFe} for the
	definitions). Indeed, if we reparametrize the curve $a_{\varepsilon(\alpha(g))gh}$ and define
	the curve $\bar{a}_{\varepsilon(\alpha(g))g}: [0,1] \to AG$ by $\bar{a}_{\varepsilon(\alpha(g))g}(s)=ha_{\varepsilon(\alpha(g))gh}(sh)$, and
	similarly we reparametrize $\sigma_{\varepsilon(\alpha(g))gh}$ and define
	$\bar{\sigma}_{\varepsilon(\alpha(g))g}: [0,1] \to G$ by $\bar{\sigma}_{\varepsilon(\alpha(g))g}(s)=\sigma_{\varepsilon(\alpha(g))gh}(sh)$,
	then the curve $\bar{a}_{\varepsilon(\alpha(g))g}$ is an $AG$-path in the $AG$-homotopy
	class defined by the element   $g\in G$. Indeed, it is clear that
	$$\bar{a}_{\varepsilon(\alpha(g))g}(t)=Tl_{\bar{\sigma}_{\varepsilon(\alpha(g(t)))g(t)^{-1}}}(\frac{d \bar{\sigma}_{\varepsilon(\alpha(g))gh}}{dt}_{|t})\; ,$$ and
	that $\bar{\sigma}_{\varepsilon(\alpha(g))g}(0)=\varepsilon(\alpha(g))$ and $\bar{\sigma}_{\varepsilon(\alpha(g))g}(1)=g$.
\end{remark}
We will denote by
\[
R^{e^-}_h: U \subseteq G \to {\mathcal U}\subseteq AG \mbox{ and } R^{e^+}_h: U \subseteq G \to \Phi^{\Gamma}_h({\mathcal U})\subseteq AG
\]
the inverse maps of the diffeomorphisms $exp^{\Gamma}_h: {\mathcal U} \to U$  and $exp^{\Gamma}_h \circ \Phi^{\Gamma}_{-h}: \Phi^{\Gamma}_{h}({\mathcal U}) \to U$, respectively. They are the exact retraction maps associated with $\Gamma$ at $h$.

Note that
\begin{equation}\label{retraction-flow1}
R^{e^-}_h(g) = a_{\varepsilon(\alpha(g))gh}(0) = \dot{\sigma}_{\varepsilon(\alpha(g))gh}(0)
\end{equation}
and
\begin{equation}\label{retraction-flow2}
R^{e^+}_h(g) = \Phi^{\Gamma}_h(R^{e^-}_h(g)) = a_{\varepsilon(\alpha(g))gh}(h) = T_g l_{g^{-1}}(\dot{\sigma}_{\varepsilon(\alpha(g))gh}(h)).
\end{equation}
The following diagram illustrates the situation
\[
\xymatrix{ {\mathcal U}\subseteq AG\ar@/^/[rrr]^{exp^{\Gamma}_{h}}
	\ar[dd]_{\Phi_{h}^\Gamma} &   & &U\subseteq G \ar@/^/[lll]^{R^{e^-}_{h}}\ar[ddlll]^{{R^{e^+}_{h}}}\\
	&  & &\\
	\Phi^{\Gamma}_{h}({\mathcal U})  \subseteq AG&  & & }
\]

\subsection{The general case}

In the general case, when we have a general Lie algebroid $(E, \lcf\; ,\rcf, \rho)$, it is possible to construct a local Lie groupoid $G$ integrating this Lie algebroid. This groupoid is local in the sense that the product is not necessarily defined on $G_2$, but only locally defined near the identity section  (see \cite{CrFe} for details). In any case, Theorem \ref{convexity-definitivo} is a local result for points near of the identities and, therefore,  it remains valid for general Lie algebroids.

\section{Homogeneous quadratic second order differential equations}\label{sec5}
\subsection{The standard case}\label{standard-case-homogeneous}

Let $\Gamma$ be a \sode\ on $TQ$ and $\Delta$ the Euler vector field on $TQ$. As we know, the flow of $\Delta$ is
\[
\Phi^{\Delta}: \mathbb{R} \times TQ \to TQ, \; \; (t, v_q) \in \mathbb{R} \times T_qQ \to \Phi^{\Delta}_t(v_q) = e^{t} v_q \in T_qQ.
\]
Then, $\Gamma$ is said to be a homogeneous quadratic \sode\ if
\begin{equation}\label{eqprima}
[\Delta, \Gamma] = \Gamma
\end{equation}
(see, for instance, \cite{LeRo}).

From (\ref{eqprima}), we have that the vector field $\Gamma$ 
is homogeneous of weight 1 with respect to $\Delta$. Anyway, in this paper, 
we will use the terminology, homogeneous quadratic $\sode$, for the following reason.
If the local expression of $\Gamma$ is
\[
\Gamma(q, \dot{q}) = \dot{q}^{i} \frac{\partial}{\partial q^{i}} + \Gamma^{i}(q, \dot{q}) \frac{\partial}{\partial \dot{q}^{i}}
\]
then, using that $\displaystyle \Delta = \dot{q}^{i}\frac{\partial}{\partial \dot{q}^{i}}$, it follows that
\[
\Delta(\Gamma^{i}) = 2\Gamma^{i}, \; \; \mbox{ for every } i,
\]
or, equivalently,
\[
\Gamma^{i}(q, \dot{q}) = \Gamma^{i}_{jk}(q)\dot{q}^j\dot{q}^{k}, \; \; \mbox{ for every } i.
\]
In other words, the (local) function $\Gamma^{i}$ is a fiberwise homogeneous quadratic polynomial function, for every $i$. In particular, this implies that $\Gamma$ vanishes along the zero section of $\tau_Q: TQ \to Q$. Thus, the trajectory of $\Gamma$ with initial velocity $0_q \in T_qQ$ ($q \in Q$) is the constant curve
\[
c_{0_q}: \mathbb{R} \to Q, \; \; t \in \mathbb{R} \to c_{0_q}(t) = q \in Q.
\]
On the other hand, if 
\[
c_{v_q}: I \to Q, \; \; t \in I \to c_{v_q}(t) \in Q,
\]
is the trajectory of $\Gamma$ with initial velocity $v_q \in T_qQ$ and $s$ is a sufficiently small real number, then the trajectory of $\Gamma$ with initial velocity $sv_{q}$ is the homothetic reparametrization of $c_{v_q}$ given by
\[
c_{sv_q}: J \to Q, \; \; u \in J \to c_{sv_q}(u) = c_{v_q}(su) \in Q.
\]
In particular, 
\begin{equation}\label{homogeneity}
c_{v_q}(s) = c_{sv_q}(1).
\end{equation}
Thus, if $h$ is a sufficiently small positive number and $v_q \in T_qQ$ then
\[
v_q \in D^{\Gamma}_{(h, q)} \Rightarrow hv_q \in D^{\Gamma}_{(1, q)}.
\]
Therefore, in what follows, we will consider the starshaped open subset $D^{\Gamma}_{(1, q)}$ of $T_qQ$ about $0_q \in T_qQ$ given by
\[
D^{\Gamma}_{(1, q)} = \{v_q \in T_qQ / (1, v_q) \in D^{\Gamma} \}.
\]
As in the general case of an standard \sode\, we will denote by $D^{\Gamma}_1$ the open subset of $TQ$ defined by
\[
D^{\Gamma}_1 = \cup_{q \in Q}D^{\Gamma}_{(1, q)}.
\]
$D^{\Gamma}_1$ is an starshaped open subset of $TQ$ about the zero section $0: Q \to TQ$ of $TQ$ and we have the exponential map of $\Gamma$ at $h = 1$, which we will simply denote by $exp^{\Gamma}: D^{\Gamma}_1 \subseteq TQ \to Q \times Q$, given by
\[
v_{q'} \in D^{\Gamma}_1 \cap T_{q'}Q \to exp^{\Gamma}(v_{q'}) = (\tau_Q(v_{q'}), exp^{\Gamma}_{\tau_Q(v_{q'})}(v_{q'})) = (q', exp^{\Gamma}_{q'}(v_{q'})) \in Q \times Q.
\]
In the particular case of a homogeneous quadratic \sode\, this map has some additional properties. In fact, using (\ref{homogeneity}), it follows that
\begin{equation}\label{+homogeneity}
exp^{\Gamma}(0_{q'}) = (q', q'), \; \; exp^{\Gamma}(tv_{q'}) = (q', c_{v_{q'}}(t)),
\end{equation}
for $q' \in \tau_Q(D^{\Gamma}_1)$, $v_{q'} \in D^{\Gamma}_1$ and $t \in [0, 1]$. 

On the other hand, as we know, the linear map
\[
T_{q'}Q \times T_{q'}Q \to T_{0_{q'}}(TQ), \; \; (u_{q'}, v_{q'}) \to ((T_{q'}0)(u_{q'}), \frac{d}{dt}_{|t =0}(tv_{q'})) 
\]
is an isomorphism. So, we can identify the tangent space $T_{0_{q'}}(D^{\Gamma}_1) = T_{0_{q'}}(TQ)$ with the product $T_{q'}Q \times T_{q'}Q$. Under this identification and using (\ref{+homogeneity}), we deduce that
\[
T_{0_{q'}}exp^{\Gamma}: T_{0_{q'}}(D^{\Gamma}_1) \simeq T_{q'}Q \times T_{q'}Q \to T_{q'}Q \times T_{q'}Q
\]
is just the identity map.

Therefore, if $q$ is a fixed point of $Q$, this implies that there exists an starshaped open subset ${\mathcal U}$ of $TQ$ about the restriction to $\tau_Q({\mathcal U})$ of the zero section, with $q \in \tau_Q({\mathcal U})$, such that the exponential map
\[
exp^{\Gamma}: {\mathcal U} \subseteq TQ \to \tau_Q({\mathcal U}) \times \tau_Q({\mathcal U}),
\]
is a diffeomorphism.

\subsection{The general case}

Let $G$ be a Lie groupoid over $Q$ and $\Gamma$ a \sode\ on the Lie algebroid $\tau: AG \to Q$. Then, following the previous section, we can introduce in a natural way the notion of a homogeneous quadratic \sode\ on $AG$.
\begin{definition} 
	$\Gamma$ is said to be a homogeneous quadratic \sode\ on $AG$ if
	\[
	[\Delta, \Gamma] = \Gamma,
	\]
	where $\Delta$ is the Euler vector field of $AG$ with global flow
	\[
	\Phi^{\Delta}: \mathbb{R} \times AG \to AG, \; \; (t, a) \to \Phi^{\Delta}(t, a) = e^{t}a.
	\]
\end{definition}
Let $(q^{i})$ be local coordinates on $Q$, $\{e_\alpha\}$ a local basis of $\Gamma(AG)$ and $(q^{i}, y^\alpha)$ the corresponding local coordinates on $AG$. If the local expression of the \sode\ $\Gamma$ on $AG$ is
\[
\Gamma(q, y) = \rho^{i}_\alpha(q) y^\alpha \frac{\partial}{\partial q^{i}} + \Gamma^{\alpha}(q, y) \frac{\partial}{\partial y^\alpha},
\]
then, using that 
\[
\Delta = y^\alpha \frac{\partial}{\partial y^\alpha},
\]
it follows that $\Gamma$ is homogeneous quadratic if and only if $\Gamma^\alpha$ is a fiberwise homogeneous quadratic function on $AG$, for every $\alpha$. This means that
\[
\Gamma^\alpha(q, y) = \Gamma^\alpha_{\beta \gamma}(q) y^\beta y^\gamma.
\]
Next, we will prove that the unique \sode\ on $V\alpha$, which is $\Psi$-related with a homogeneous quadratic \sode\ on $AG$, also is homogeneous quadratic.
\begin{proposition}\label{homogeneous-quadratic-related}   
	Let $\Gamma$ be a homogeneous quadratic \sode\ on $AG$ and $\tilde{\Gamma}$ the unique \sode\ on $V\alpha$ which is $\Psi$-related with $\Gamma$. Then, $\tilde{\Gamma}$ also is homogeneous quadratic. 
\end{proposition}
\begin{proof}
	We will proceed as follows: 
	\begin{enumerate}
		\item
		We will see that the Lie bracket $[\tilde{\Delta}, \tilde{\Gamma}]$ is a \sode\ on the Lie algebroid $V\alpha$, where $\tilde{\Delta}$ is the Euler vector field of the Lie algebroid $\tilde{\tau}: V\alpha \to G$.
		\item
		We will prove that the \sode\ $[\tilde{\Delta}, \tilde{\Gamma}]$ is $\Psi$-related with $\Gamma$.
	\end{enumerate}
	Thus, from Proposition \ref{rel-SODES}, we will conclude that the result holds.
	
	(i) Let $\tilde{f}$ be a real $C^{\infty}$-function on $G$. Then, using (\ref{SODE-2}) and the fact that $\tilde{\Delta}$ is a $\tilde{\tau}$-vertical vector field, we have that
	\[
	[\tilde{\Delta}, \tilde{\Gamma}](\tilde{f} \circ \tilde{\tau}) = \tilde{\Delta}(\tilde{\Gamma}(\tilde{f} \circ \tilde{\tau})) - \tilde{\Gamma}(\tilde{\Delta}(\tilde{f} \circ \tilde{\tau})) = \tilde{\Delta}(\widehat{d^{V\alpha}\tilde{f}}).
	\]
	On the other hand, since the Lie derivative with respect to $\tilde{\Delta}$ of a fiberwise linear function on $V\alpha$ is just the same function, we deduce that
	\[
	[\tilde{\Delta}, \tilde{\Gamma}](\tilde{f} \circ \tilde{\tau}) = \widehat{d^{V\alpha}\tilde{f}}.
	\]
	This, from (\ref{SODE-2}), implies (i).
	
	(ii) Using that $\Psi$ is a vector bundle morphism, it follows that
	\[
	\Psi \circ \Phi^{\tilde{\Delta}}_t = \Phi^{\Delta}_t \circ \Psi,
	\]
	and, therefore, $\tilde{\Delta}$ and $\Delta$ are $\Psi$-related. So, since that $\tilde{\Gamma}$ and $\Gamma$ also are $\Psi$-related, we conclude that the Lie brackets $[\tilde{\Delta}, \tilde{\Gamma}]$ and  $[\Delta, \Gamma]$ also are $\Psi$-related. So, using that $[\Delta, \Gamma] = \Gamma$, we deduce the result.
\end{proof}
Under the same hypotheses as in Proposition \ref{homogeneous-quadratic-related}, we will denote by
\[
exp^{\Gamma}_h: {\mathcal U} \subseteq AG \to U\subseteq G
\]
the exponential map associated with $\Gamma$ at $h > 0$ and for the point $q_0 \in Q$ as in Theorem \ref{convexity-definitivo}, that is,
\[
exp^{\Gamma}_h(v) = \tilde{\tau}(\Phi^{\tilde{\Gamma}}_h(v)), \; \; \mbox{ for } v \in {\mathcal U}.
\]
Using that $\tilde{\Gamma}$ is a homogeneous quadratic \sode\, we can take $h = 1$ and ${\mathcal U}$ an starshaped open subset of $AG$ about the restriction of the zero section to $\tau({\mathcal U}) \subseteq Q$. We will denote by
\[
exp^{\Gamma}: {\mathcal U} \subseteq AG \to U \subseteq G
\]
the corresponding map and we will prove the following result.
\begin{proposition}
	If $exp^{\Gamma}: {\mathcal U} \subseteq AG \to U\subseteq G$ is the exponential map associated with a homogeneous quadratic \sode\ $\Gamma$ on $AG$, we have that
	\begin{equation}\label{Properties-homogeneous-quadratic-SODES}
	exp^{\Gamma}(0_q) = \varepsilon(q), \; \; exp^{\Gamma}(tv_{\varepsilon(q)}) = \sigma_{v_{\varepsilon(q)}}(t),
	\end{equation}
	for $q \in \tau({\mathcal U})$, $v_{\varepsilon(q)} \in {\mathcal U} \cap A_qG$ and $t \in [0, 1]$, where
	\[
	\sigma_{v_{\varepsilon(q)}}: [0, 1] \to U \subseteq G
	\]
	is the trajectory of $\tilde{\Gamma}$ with initial velocity $v_{\varepsilon(q)}$. Moreover, under the canonical identifications,
	\[
	T_{0_q}exp^{\Gamma}: T_{0_q}({\mathcal U}) \simeq T_qQ \times A_qG \to T_{\varepsilon(q)}G \simeq T_qQ \times A_qG
	\]
	is just the identity map.
\end{proposition} 
\begin{proof}
	The \sode\ $\tilde{\Gamma}$ on the Lie algebroid $\tilde{\tau}: V \alpha \to G$ may be considered as a smooth family of standard \sode s, one on each fiber of the source map $\alpha: G \to Q$. In addition, each one of these \sode s is homogeneous quadratic. So, using the results in Section \ref{standard-case-homogeneous}, we deduce that (\ref{Properties-homogeneous-quadratic-SODES}) holds.
	
	Now, the linear map 
	\[
	T_qQ \times A_qG \to T_{0_q}AG, \; \; (u_q, v_{\varepsilon(q)}) \to (T_q0)(u_q) + \frac{d}{dt}_{|t=0}(t v_{\varepsilon(q)})
	\]
	is a linear isomorphism, where $0: Q \to AG$ is the zero section. Thus, we can identify $T_{0_q}{\mathcal U} = T_{0_q}AG$ 
	with the product space $T_qQ \times A_qG$.
	
	On the other hand, the linear map
	\[
	T_qQ \times A_qG \to T_{\varepsilon(q)}G, \; \; (u_q, v_{\varepsilon(q)}) \to (T_q\varepsilon)(u_q) + \frac{d}{dt}_{|t=0}(t v_{\varepsilon(q)})
	\]
	also is a linear isomorphism and we can identify $T_{\varepsilon(q)}G$ with the same product space $T_qQ \times A_qG$. 
	
	Under the previous identifications and using (\ref{Properties-homogeneous-quadratic-SODES}), we conclude that 
	\[
	T_{0_q}exp^{\Gamma}:  T_{0_q}({\mathcal U}) \simeq T_qQ \times A_qG \to T_{\varepsilon(q)}G \simeq T_qQ \times A_qG
	\]
	is the identity map.
\end{proof}

\section{Conclusions and future work}\label{conclusions-future-work}

We have developed a local convexity theory for a \sode\ $\Gamma$ on the Lie algebroid $AG$ of a Lie groupoid $G$. In fact, we introduce the exponential map associated with $\Gamma$ as a local diffeomorphism between $AG$ and $G$. The particular case when $\Gamma$ is homogeneous quadratic is discussed in the last part of the paper.

Now, in the presence of a mechanical continuous Lagrangian function $L: AG \to \mathbb{R}$, we can consider the \sode\ $\Gamma_L$ on $AG$ whose trajectories are the solutions of the Euler-Lagrange equations for $L$ (see \cite{LeMaMa,Ma0}). Then, we can apply the results in this paper to $\Gamma_L$ and we can introduce, in a natural way, a discrete Lagrangian function $\mathbb{L}_h^{e}: G \to \mathbb{R}$ on $G$, with $h$ a sufficiently small positive real number. $\mathbb{L}_h^{e}$ is the exact discrete Lagrangian function associated with $L$ (see \cite{MaMaMa3,MaMaMa4}). In addition, if we take a discrete Lagrangian function $L_d: G \to \mathbb{R}$ which is an approximation of order $r$ of $\mathbb{L}_h^{e}$, then one may prove that the discrete scheme induced by $L_d$ is an approximation of the continuous flow $\Gamma_L$ of order $r$ (see \cite{MaMaMa3,MaMaMa4}). In other words, $\mathbb{L}_h^{e}$ is crucial to discuss the variational error analysis.

On the other hand, the results in Section \ref{section2} of this paper also play an important role in the definition of the exponential map and the exact discrete submanifold associated with an standard mechanical nonholonomic system. In turn, the previous objects allow to introduce the exact discrete nonholonomic equations (see \cite{AnMaMa}).

\section*{Acknowledgments} The first part of the proof of Theorem \ref{isomorfismo-exp-h0-q0} was proposed to two of the authors (JCM and DMdeD) by 
JC Sabina de Lis (ULL, Spain). This result is relevant for the rest of the paper. So, the authors are very grateful for the invaluable comments of JC Sabina de Lis. JCM acknowledges financial support from
the IUMA (University of Zaragoza, Spain) and the Spanish Ministry of Science and Innovation under grant PGC2018-098265-B-C32. D. Mart{\'\i}n de Diego acknowledges financial support from
the Spanish Ministry of Science and Innovation, under grant PID2019-
106715GB-C21 and  from the Spanish Ministry of Science and Innovation, through the ``Severo Ochoa Programme for Centres of Excellence in R\&D" (CEX2019-000904-S). EMF acknowledges financial support from
the Spanish Ministry of Science and Innovation under grant PGC2018-098265-B-C31.

\appendix
\section{Proofs of Theorems \ref{convexity1} and \ref{isomorfismo-exp-h0-q0}} 
\label{Hartmann}

In this appendix, we will give a proof of Theorems \ref{convexity1} and \ref{isomorfismo-exp-h0-q0}.
For this purpose, we will use some standard results on second order differential equations on $\R^n$ (see \cite{Ha}).

Let
\[
\displaystyle \frac{d^2q^{i}}{dt^2} = \xi^{i}(t, q^j, \frac{dq^j}{dt}), \; \; \forall i \in \{1, \dots, n\}
\]
be a system of second order differential equations on $\R \times \R^n$, with $\xi^{i}$ a real $C^{\infty}$-function on a compact subset of $\R \times \R^{2n}$ which contains the origin.

We will consider the problem of the existence of solutions satisfying the boundary conditions
\[
q^{i}(0) = 0, \; \; \; q^{i}(h_0) = 0, \; \; \forall i, \mbox{ with } h_0 > 0.
\]
In this direction, using Corollary 4.1 of  Chapter XII in \cite{Ha}, we deduce the following result.
\begin{theorem}\label{Hartman} 
	Let $\xi^{i}(t, q, \dot{q})$ be continuous for $1 \leq i \leq n$, $0 \leq t \leq h_0$ ($h_0 > 0$), $\|q\| \leq R$, $\|\dot{q}\| \leq \dot{R}$
	such that $f$ satisfies a Lipschitz condition with respect to $q, \dot{q}$ of the form
	\[
	\| \xi(t, q_1^j, \dot{q}^j_1) - \xi(t, q_2^j, \dot{q}_2^j)\| \leq C \| q_2 - q_1\| + \dot{C} \| \dot{q}_2 - \dot{q}_1 \|
	\]
	with Lipschitz constants $C, \dot{C}$, so small that
	\[
	\frac{C h_0^2}{8} + \frac{\dot{C} h_0}{2} < 1. 
	\]
	In addition, suppose that $\| \xi(t, q^j, \dot{q}^j) \| \leq M$ and that
	\[
	\frac{M h_0^2}{8} \leq R, \; \; \; \frac{M h_0}{2} \leq \dot{R}.
	\]
	Then, the system of second order differential equations
	\[
	\frac{d^2 q^{j}}{dt^2} = \xi^{j} (t, q^{i}, \dot{q}^{i}), \; \; \mbox{ for all } j
	\]
	has a unique solution satisfying
	\[
	\|q(t)\| \leq R, \; \; \|\dot{q}(t)\| \leq \dot{R}, \; \; q^{i}(0) = 0, \; \; q^{i}(h_0) = 0, \; \; \mbox{ for all } t \in [0, h_0] \mbox{ and } 1 \leq i \leq n.
	\]
\end{theorem}

We will also use the following classical result.

\begin{proposition}\label{Dif-acotada-Lipschitz}
	Let $f: U \subseteq \mathbb{R}^n \to \mathbb{R}^n$ be a $C^{\infty}$-smooth map, with $U$ a convex open subset of $\mathbb{R}^n$ and suppose that there exists a positive constant $C > 0$ such that
	\[
	\|df(x)\| \leq C, \; \; \forall x \in U.
	\]
	Then, we have that
	\[
	\|f(x) - f(y) \| \leq C \|x - y\|, \; \; \mbox{ for } x, y \in U.
	\]
\end{proposition}

Now, we may prove Theorem \ref{convexity1}.

\begin{proof}({\it proof of Theorem \ref{convexity1}})
	
	Let $(\tilde{U}, \tilde{\varphi} \equiv (q^i))$ be a local chart
	on $Q$ such that
	\[
	\tilde{\varphi}(\tilde{U}) = B(0; \epsilon) \; \mbox{ and }
	\tilde{\varphi}(q_0) = (0, \dots , 0),
	\]
	where $B(0; \epsilon)$ is the open ball in $\mathbb{R}^n$ of
	center the origin and radius $\epsilon > 0$.
	
	We consider the corresponding local coordinates
	$(\tau_{Q}^{-1}(\tilde{U}), \bar{\varphi} \equiv (q^i, \dot{q}^i))$ on
	$TQ$. Note that $\bar{\varphi}(\tau_{Q}^{-1}(\tilde{U})) =
	\tilde{\varphi}(\tilde{U}) \times \mathbb{R}^n$. Since $\Gamma$ is
	a \sode, we also have that
	\[
	\Gamma (q, \dot{q}) = \displaystyle \dot{q}^i \frac{\partial}{\partial q^i} + \xi^{i}
	(q, \dot{q}) \frac{\partial}{\partial \dot{q}^i}.
	\]
	Then, the trajectories of $\Gamma$ in $\tilde{U}$ are the
	solutions of the system of second order differential equations
	\[
	\displaystyle \frac{d^2q^i}{dt^2} = \xi^i(q, \frac{dq}{dt}), \; \;
	\; \mbox{ for all } i.
	\]
	Now, if we take
	\[
	0 < R < \epsilon \; \; \mbox{ and } \; \; 0 < \dot{R}
	\]
	then, using that $\xi^{i}$ is a real $C^{\infty}$-function on $B(0; \epsilon) \times \mathbb{R}^n$, we deduce that there exist positive constants $C, \dot{C} > 0$ satisfying 
	\[
	\|d_1\xi(q, \dot{q})\| \leq C, \; \; \|d_2\xi(q, \dot{q})\| \leq \dot{C}, \; \; \mbox{ for } (q, \dot{q}) \in \bar{B}(0; R) \times \bar{B}(0; \dot{R}),
	\]
	where $\bar{B}(0; R)$ and $\bar{B}(0; \dot{R})$ are the closed balls in $\mathbb{R}^n$ of center the origin and radius $R$ and $\dot{R}$, respectively.
	
	Thus, from Proposition \ref{Dif-acotada-Lipschitz}, it follows that
	\[
	\begin{array}{rcl}
	\| \xi^{i}(q^j_1, \dot{q}^j_1) - \xi^{i}(q^j_2, \dot{q}^j_2)\| &\leq & \| \xi^{i}(q^j_1, \dot{q}^j_1) - \xi^{i}(q^j_2, \dot{q}^j_1)\| + \| \xi^{i}(q^j_2, \dot{q}^j_1) - \xi^{i}(q^j_2, \dot{q}^j_2)\| \\[5pt]
	& \leq & C \| q_2 - q_1\| + \dot{C} \| \dot{q}_2 - \dot{q}_1 \|
	\end{array}
	\]
	for $(q^j_1, \dot{q}^j_1), (q^j_2, \dot{q}^j_2) \in \bar{B}(0, R) \times \bar{B}(0; \dot{R})$.
	
	Moreover, it is clear that there exists a positive constant $M > 0$ and
	\[
	\| \xi(q^j, \dot{q}^j) \| \leq M, \; \; \forall (q, \dot{q}) \in \bar{B}(0; R) \times \bar{B}(0, \dot{R}).
	\]
	Next, we choose a sufficiently small positive number $h_0$ satisfying
	\[
	\displaystyle \frac{Ch_0^2}{8} + \frac{\dot{C}h_0}{2} < 1, \; \; \frac{Mh_0^2}{8} \leq R, \; \; \frac{Mh_0}{2} \leq \dot{R}.
	\]
	Now, if we take $h\in \mathbb{R}$, $0 < h \leq h_0$ and the compact subsets $C$ and $\bar{C}$ of $Q$ and $TQ$, respectively, given by
	\[
	C = \tilde{\varphi}^{-1}(\bar{B}(0; R)), \; \; \bar{C} = \bar{\varphi}^{-1}(\bar{B}(0; R) \times \bar{B}(0, \dot{R}))
	\]
	then, using Theorem \ref{Hartman}, we conclude that there exists a unique trajectory
	$\sigma_{q_{0}q_{0}h}: [0, h] \to C \subseteq Q$ of
	$\Gamma$ such that
	\[
	\sigma_{q_{0}q_{0}h}(0) = q_0, \; \; \; \sigma_{q_{0}q_{0}h}(h) =
	q_0,
	\]
	and
	\[
	\dot{\sigma}_{q_0q_0h}(t) \in \bar{C}, \; \; \mbox{ for } t \in [0, h].
	\]
	Therefore, it is sufficient to define $v_{(h, q_0)} : = \dot{\sigma}_{q_0q_0h}(0)$ and we end the proof of the result.
\end{proof}
Next, we will prove Theorem \ref{isomorfismo-exp-h0-q0}.
\begin{proof}({\it Proof of Theorem \ref{isomorfismo-exp-h0-q0}})
	
	From Theorem \ref{convexity1}, it follows that
	\[
	v_{(h, q_0)} \in D^\Gamma_{(h, q_0)} \; \; \mbox{ and } \; \; exp^{\Gamma}_{(h, q_0)}(v_{(h, q_0)}) = q_0.
	\]
	Moreover, it is clear that the map
	\[
	exp^{\Gamma}_{(h, q_0)}: D^{\Gamma}_{(h, q_0)} \subseteq T_{q_0}Q \to Q
	\]
	is smooth.
	
	Next, we will proceed locally. So, we will denote by
	\[
	(t, q^{i}, \dot{q}^{i}) \to (x^j(t, q^{i}, \dot{q}^{i}), \dot{x}^j(t, q^{i}, \dot{q}^{i}))
	\]
	the flow of the \sode\ $\Gamma$
	\[
	\Gamma(q^{j}, \dot{q}^{j}) = \dot{q}^{i}\frac{\partial}{\partial q^{i}} + \xi^{i}(q^{j}, \dot{q}^j) \frac{\partial}{\partial \dot{q}^{i}}.
	\]
	We have that
	\begin{equation}\label{second-order} 
	\ddot{x}^{i}(t, q^j, \dot{q}^j) = \xi^{i}(x^k(t, q^{j}, \dot{q}^{j}), \dot{x}^k(t, q^{j}, \dot{q}^{j})),
	\end{equation}
	and
	\begin{equation}\label{initial-conditions}
	x^{i}(0, q^j, \dot{q}^j) = q^{i}, \; \; \; \dot{x}^{i}(0, q^{j}, \dot{q}^j) = \dot{q}^{i}.
	\end{equation}
	The local expression of the map $exp^{\Gamma}_{(h, q_0)}$ is
	\[
	\dot{q}^{i} \to exp^{\Gamma}_{(h, q_0)}(\dot{q}^{i}) = x(h, q_0, \dot{q}^{i}).
	\]
	Denote by $\dot{q}_{0h}$ the tangent vector $v_{(h, q_0)} \in T_{q_0}Q$. We must prove that the Jacobian matrix of $exp^{\Gamma}_{(h, q_0)}$ at $\dot{q}_{0h}$
	\[
	(D_{\dot{q}} exp^{\Gamma}_{(h, q_0)})(\dot{q}_{0h}) = (D_{\dot{q}}x)(h, q_0, \dot{q}_{0h})
	\]
	is non-singular.
	
	Let $U_{(q_0,\dot{q}_{0h})}(t)$ be the Jacobian matrix of the smooth map $exp^{\Gamma}_{(t, q_0)}$ at $\dot{q}_{0h}$, that is,
	\[
	U_{(q_0,\dot{q}_{0h})}(t) = (D_{\dot{q}} exp^{\Gamma}_{(t, q_0)})(\dot{q}_{0h}) = (D_{\dot{q}}x)(t, q_0, \dot{q}_{0h}).
	\]
	Then, using (\ref{second-order}), an standard argument proves that
	\[
	\begin{array}{rcl}
	\ddot{U}_{(q_0, \dot{q}_{0h})}(t) & = & (D_{q}\xi)(x^{i}(t, q_0, \dot{q}_{0h}), \dot{x}^{i}(t, q_0, \dot{q}_{0h}))U_{(q_0,\dot{q}_{0h})}(t) \\ [5pt]
	& + & (D_{\dot{q}}\xi)(x^{i}(t, q_0, \dot{q}_{0h}), \dot{x}^{i}(t, q_0, \dot{q}_{0h}))\dot{U}_{(q_0,\dot{q}_{0h})}(t)
	\end{array}
	\]
	and, in a similar way using (\ref{initial-conditions}), we also deduce that
	\[
	U_{(q_0, \dot{q}_{0h})}(0) = 0, \; \; \; \dot{U}_{(q_0, \dot{q}_{0h})}(0) = Id.
	\]
	So, if we denote by $B_{(q_0, \dot{q}_{0h})}(t)$ and $F_{(q_0, \dot{q}_{0h})}(t)$ the matrices 
	\[
	(D_{q}\xi)(x^{i}(t, q_0, \dot{q}_{0h}), \dot{x}^{i}(t, q_0, \dot{q}_{0h})) \mbox{ and }  (D_{\dot{q}}\xi)(x^{i}(t, q_0, \dot{q}_{0h}), \dot{x}^{i}(t, q_0, \dot{q}_{0h})),
	\]
	respectively, it follows that
	\[
	\ddot{U}_{(q_0, \dot{q}_{0h})}(t) = B_{(q_0, \dot{q}_{0h})}(t) U_{(q_0, \dot{q}_{0h})}(t) + F_{(q_0, \dot{q}_{0h})}(t) \dot{U}_{(q_0, \dot{q}_{0h})}(t).
	\]
	Now, we consider the homogeneous system of second order differential equations
	\begin{equation}\label{homogeneous-system}
	\ddot{y}(t) = B_{(q_0, \dot{q}_{0h})}(t) y(t) + F_{(q_0, \dot{q}_{0h})}(t) \dot{y}(t).
	\end{equation}
	Note that $B_{(q_0,\dot{q}_{0h})}$ and $F_{(q_0, \dot{q}_{0h})}$ are $C^{\infty}$-matrices, for every sufficiently small positive number $h$.
	
	So, taking into account that there exists a compact subset $\bar{C} \subseteq TQ$ such that $v_{(h, q_0)} \in \bar{C}$ (for every $h$), using Theorem \ref{Hartman} and proceeding as in the proof of Theorem \ref{convexity1}, we conclude that there exists a sufficiently small positive number $p_0 > 0$ such that for all $h$ the unique solution 
	\[
	t \to y_{(q_0,\dot{q}_{0h})}(t)
	\]
	of the system (\ref{homogeneous-system}),
	satisfying the boundary conditions
	\[
	y_{(q_0,\dot{q}_{0h})}(0) = 0, \; \; y_{(q_0,\dot{q}_{0h})}(p) = 0, \; \; \mbox{ with } 0 < p \leq p_0,
	\]
	is the trivial solution.
	
	Thus, from Lemma 3.1, Chapter XII in \cite{Ha}, we deduce that the matriz 
	\[
	U_{(q_0, \dot{q}_{0h})}(p), \; \; \mbox{ with } 0 < p \leq p_0,
	\]
	is regular, for every $h$.
	
	Therefore, it is sufficient to take $h = p$, with $0 < p \leq p_0$, and the result is proved.
\end{proof}
\section{Lie algebroids and groupoids}\label{algebroide-grupoide}

First of all, we will recall the definition of a Lie groupoid
and some generalities about them are explained (for more details,
see \cite{Mac}).

A groupoid over a set $M$ is a set $G$ together with the
following structural maps:
\begin{itemize}
	\item A pair of maps $\alpha: G \to M$, the source, and
	$\beta: G \to M$, the target. Thus, an element $g \in G$ is
	thought as an arrow from $q_0= \alpha(g)$ to $q_1 = \beta(g)$ in $M$
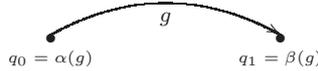
	$$
	\xymatrix{*=0{\stackrel{\bullet}{\mbox{\tiny
					$q_0=\alpha(g)$}}}{\ar@/^1pc/@<1ex>[rrr]_g}&&&*=0{\stackrel{\bullet}{\mbox{\tiny
					$q_1=\beta(g)$}}}}
	$$
	The maps $\alpha$ and $\beta$ define the set of composable pairs
	$$
	G_{2}=\{(g,h) \in G \times G / \beta(g)=\alpha(h)\}.
	$$
	\item A multiplication $m: G_{2} \to G$, to be denoted simply by $m(g,h)=gh$, such that
	\begin{itemize}
		\item $\alpha(gh)=\alpha(g)$ and $\beta(gh)=\beta(h)$.
		\item $g(hk)=(gh)k$.
	\end{itemize}
	If $g$ is an arrow from $q_0 = \alpha(g)$ to $q_1 = \beta(g)
	$ and $h$ is an arrow from $q_1=\beta(g) = \alpha(h)$ to $q_2 = \beta(h)$ then
	$gh$ is the composite arrow from $q_0$ to $q_2$
	$$
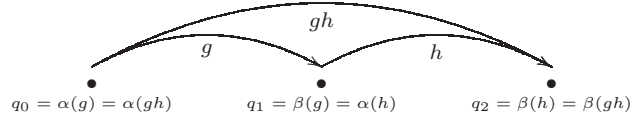
\xymatrix{*=0{\stackrel{\bullet}{\mbox{\tiny
					$q_0=\alpha(g)=\alpha(gh)$}}}{\ar@/^2pc/@<2ex>[rrrrrr]_{gh}}{\ar@/^1pc/@<2ex>[rrr]_g}&&&*=0{\stackrel{\bullet}{\mbox{\tiny
					$q_1=\beta(g)=\alpha(h)$}}}{\ar@/^1pc/@<2ex>[rrr]_h}&&&*=0{\stackrel{\bullet}{\mbox{\tiny
					$q_2=\beta(h)=\beta(gh)$}}}}$$
	\item An identity map $\varepsilon: Q \to G$, a section of $\alpha$ and $\beta$, such that
	\begin{itemize}
		\item $\varepsilon(\alpha(g))g=g$ and $g\varepsilon(\beta(g))=g$.
	\end{itemize}
	\item An inversion map $i: G \to G$, to be denoted simply by $i(g)=g^{-1}$, such that
	\begin{itemize}
		\item $g^{-1}g=\varepsilon(\beta(g))$ and $gg^{-1}=\varepsilon(\alpha(g))$.
	\end{itemize}
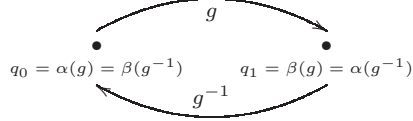
	$$\xymatrix{*=0{\stackrel{\bullet}{\mbox{\tiny
					$q_0=\alpha(g)=\beta(g^{-1})$}}}{\ar@/^1pc/@<2ex>[rrr]_g}&&&*=0{\stackrel{\bullet}{\mbox{\tiny
					$q_1=\beta(g)=\alpha(g^{-1})$}}}{\ar@/^1pc/@<2ex>[lll]_{g^{-1}}}}$$
	
\end{itemize}

A groupoid $G$ over a set $Q$ will be denoted simply by the symbol
$G \rightrightarrows Q$.

The groupoid $G \rightrightarrows Q$ is said to be a Lie
groupoid if $G$ and $Q$ are manifolds and all the structural maps
are differentiable with $\alpha$ and $\beta$ differentiable
submersions. If $G \rightrightarrows Q$ is a Lie groupoid then $m$
is a submersion, $\varepsilon$ is an immersion and $i$ is a
diffeomorphism. Moreover, if $q \in Q$, $\alpha^{-1}(q)$ (resp.,
$\beta^{-1}(q)$) will be said the $\alpha$-fiber (resp.,
the $\beta$-fiber) of $q$.

Typical examples of Lie groupoids are: the pair or banal groupoid $Q \times Q$ over $Q$, a Lie group $G$ (as a Lie groupoid over a single point), the Atiyah groupoid $(Q \times Q)/G$ (over $Q/G$) associated with a free and proper action of a Lie group $G$ on $Q$ and the Lie groupoid $G\pi$ associated with a fibration $\pi: Q \to M$ given by
\[
G\pi = \{(q, q') \in Q \times Q / \pi(q) = \pi(q') \},
\]
(it is a Lie subgroupoid of the pair groupoid $Q \times Q \rightrightarrows Q$).    

On the other hand, if $G \rightrightarrows Q$ is a Lie groupoid and $g \in G$ then the left-translation by
$g \in G$ and the right-translation by $g$ are the
diffeomorphisms
$$
\begin{array}{lll}
l_{g}: \alpha^{-1}(\beta(g)) \longrightarrow
\alpha^{-1}(\alpha(g))&; \; \;& h \longrightarrow
l_{g}(h) = gh, \\
r_{g}: \beta^{-1}(\alpha(g)) \longrightarrow
\beta^{-1}(\beta(g))&; \; \;& h \longrightarrow r_{g}(h) = hg.
\end{array}
$$
Note that $l_{g}^{-1} = l_{g^{-1}}$ and $r_{g}^{-1} = r_{g^{-1}}$.

A vector field $\tilde{X}$ on $G$ is said to be
left-invariant (resp., right-invariant) if it is
tangent to the fibers of $\alpha$ (resp., $\beta$) and
$\tilde{X}(gh) = (T_{h}l_{g})(\tilde{X}(h))$ (resp.,
$\tilde{X}(gh) = (T_{g}r_{h})(\tilde{X}(g)))$, for $(g,h) \in
G_{2}$.

Now, we will recall the definition of the Lie algebroid
associated with $G$.

We consider the vector bundle $\tau: AG \to Q$, whose fiber at a
point $q \in Q$ is $A_{q}G = V_{\varepsilon(q)}\alpha = Ker
(T_{\varepsilon(q)}\alpha)$. It is easy to prove that there exists a
bijection between the space $\Gamma (\tau)$ of sections of $\tau: AG \to Q$ and the set of
left-invariant (resp., right-invariant) vector fields on $G$. If
$X$ is a section of $\tau: AG \to Q$, the corresponding
left-invariant (resp., right-invariant) vector field on $G$ will
be denoted by $\lvec{X}$ (resp., $\rvec{X}$), where
\begin{equation}\label{linv}
\lvec{X}(g) = (T_{\varepsilon(\beta(g))}l_{g})(X(\beta(g))),
\end{equation}
\begin{equation}\label{rinv}
\rvec{X}(g) = -(T_{\varepsilon(\alpha(g))}r_{g})((T_{\varepsilon
	(\alpha(g))}i)( X(\alpha(g)))),
\end{equation}
for $g \in G$. Using the above facts, one may introduce a 
bracket $\lcf\cdot , \cdot\rcf$ on the space of sections $\Gamma(AG)$ and a bundle map $\rho: AG \to TQ$, which
are defined by
\begin{equation}\label{LA}
\lvec{\lcf X, Y\rcf} = [\lvec{X}, \lvec{Y}], \makebox[.3cm]{}
\rho(X)(q) = (T_{\varepsilon(q)}\beta)(X(q)),
\end{equation}
for $X, Y \in \Gamma(AG)$ and $q \in Q$. 

Since $[\cdot, \cdot]$ induces a Lie algebra structure on the space of vector fields on $G$, it is easy to prove that $\lcf \cdot, \cdot \rcf$ also defines a Lie algebra structure on $\Gamma(AG)$. In addition, it follows that
\[
\lcf X, f Y\rcf = f \lcf X, Y\rcf + \rho(X)(f) Y,
\]
for $X, Y \in \Gamma(AG)$ and $f \in C^{\infty}(M)$.

In other words, we have a Lie algebroid structure on the vector bundle $\tau: AG \to M$ with anchor map $\rho$.
It is the Lie algebroid of $G$.

We remark the following facts:
\begin{itemize} 
	\item
	The Lie algebroid of the pair groupoid $Q \times Q$ over $Q$ is the standard Lie algebroid $TQ \to Q$;
	
	\item
	The Lie algebroid of a Lie group $G$ is the Lie algebra ${\mathfrak g}$ of $G$;
	
	\item
	The Lie algebroid of the Atiyah groupoid $(Q \times Q)/G$ is the Atiyah algebroid $TQ/G$ over $Q/G$ and
	
	\item
	The Lie algebroid of the Lie groupoid $G\pi$ associated with a fibration $\pi: Q \to M$ is the vertical bundle $V\pi$ of $\pi: Q \to M$.  
\end{itemize}

\end{document}